\documentclass[a4paper,reqno,11pt]{amsart}

\usepackage{amsmath}
\usepackage{amssymb}
\usepackage{cases}
\usepackage{enumitem}
\allowdisplaybreaks[4]

\input xy
\xyoption{all}

\newtheorem{thm}{Theorem}[section]
\newtheorem{prop}[thm]{Proposition}

\newtheorem{lem}[thm]{Lemma}

\newtheorem{cor}[thm]{Corollary}

\newtheorem{conj}[thm]{Conjecture}
\newtheorem{claim}[thm]{Claim}

\theoremstyle{definition}
\newtheorem{definition}[thm]{Definition}

\theoremstyle{remark}
\newtheorem{remark}[thm]{Remark}

\numberwithin{equation}{section}

\newcommand{\bQ}{\mathbb{Q}}
\newcommand{\bP}{\mathbb{P}}

\newcommand\OO{{\mathcal{O}}}

\newcommand{\roundup}[1]{\lceil{#1}\rceil}
 
\newcommand{\bR}{{\mathbb R}}

\newcommand\Vol{\text{\rm Vol}}

\newcommand\Nlc{{\rm{Nlc}}}

\newcommand\Supp{{\rm{Supp}}}

\newcommand\mld{{\rm{mld}}}

\begin{document}

\title{Birational boundedness of rationally connected Calabi--Yau 3-folds}
\date{\today}

\author{Weichung Chen}
\address{Graduate School of Mathematical Sciences, the University of Tokyo, Tokyo 153-8914, Japan.}
\email{chenweic@ms.u-tokyo.ac.jp}

\author{Gabriele Di Cerbo}
\address{Department of Mathematics, Princeton University, Princeton, NJ 08540, USA}
\email{dicerbo@math.princeton.edu}

\author{Jingjun Han}
\address{Department of Mathematics, Johns Hopkins University, Baltimore, MD 21218, USA}
\email{jhan@math.jhu.edu}

\author{Chen Jiang}
\address{Shanghai Center for Mathematical Sciences, Fudan University, Jiangwan Campus, 2005 Songhu Road, Shanghai, 200438, China}
\email{chenjiang@fudan.edu.cn}

\author{Roberto Svaldi}
\address{EPFL, SB MATH-GE, MA B1 497 (B\^{a}timent MA), Station 8, CH-1015 Lausanne, Switzerland}
\email{roberto.svaldi@epfl.ch}

\begin{abstract}
We prove that rationally connected Calabi--Yau 3-folds
with kawamata log terminal (klt) singularities form a birationally bounded family, or more generally, rationally connected $3$-folds of $\epsilon$-CY type form a birationally bounded family for $\epsilon>0$. Moreover, we show that the set of $\epsilon$-lc log Calabi--Yau pairs $(X, B)$ with coefficients of $B$ bounded away from zero is log bounded modulo flops. As a consequence, we deduce that rationally connected klt Calabi--Yau $3$-folds with mld bounded away from $1$ are bounded modulo flops. 
\end{abstract}

\keywords{Calabi--Yau 3-folds, boundedness, rationally connected}
\subjclass[2010]{14J30, 14M22, 14E30}
\maketitle
\pagestyle{myheadings} \markboth{\hfill  W. Chen, G. Di Cerbo, J. Han,
C. Jiang, and R. Svaldi
\hfill}{\hfill Birational boundedness of rationally connected Calabi--Yau 3-folds \hfill}


\section{Introduction}
Throughout this paper, we work over an uncountable algebraically closed field
of characteristic 0, for instance, the complex number field $\mathbb{C}$.

A normal projective variety $X$ is a {\it Fano} (resp. {\it Calabi--Yau}) variety if $-K_X$ is ample (resp. $K_X\equiv 0$).
According to the Minimal Model Program, Fano varieties and Calabi--Yau varieties, along with varieties of general type, form fundamental classes in birational geometry as building blocks of algebraic varieties.
Hence, it is interesting to ask whether such kinds of varieties satisfy any finiteness properties, namely, whether they can be parametrized by finitely many algebraic families.
In this regard, Birkar \cite{Bir16a, Bir16b} recently showed that the set of $\epsilon$-lc Fano varieties of dimension $d$ forms a bounded family for fixed $\epsilon>0$ and $d$. This is known as the Borisov--Alexeev--Borisov (BAB) Conjecture.

However, Calabi--Yau varieties in general are not bounded in the category of algebraic varieties: for example, it is well-known that there are infinitely many algebraic families of projective K3 surfaces. Nonetheless, K3 surfaces all fit into a unique topological family, once we consider also the non-algebraic ones. A similar picture holds for abelian varieties. In higher dimension, the situation is even more varied and there are classes of Calabi--Yau varieties for which boundedness is still a hard unresolved question, for example, finiteness of topological types of smooth Calabi--Yau $3$-folds (in the strict sense) has been open for over 40 years and is a pivotal question in string theory.

Therefore, in this article, rather than considering Calabi--Yau varieties in full generality, we focus on a special class: that of rationally connected Calabi--Yau varieties.
Recall that a variety is {\it rationally connected} if any two general points can be connected by a rational curve. 
Recall that all klt Fano varieties are rationally connected by \cite{ZQ, hm07}. Hence, rationally connected Calabi--Yau varieties can be viewed as those Calabi--Yau varieties which behave most like Fano varieties. This class of varieties has received scarce attention so far due to the technical difficulties involved in its treatment. The recent developments in the study of boundedness of algebraic varieties, e.g., \cite{HMX14, Jiang6, Bir16b, Bir18, Bir16a, FS20,F20,Bir20, BDCS}, provide
new tools to approach the issue for this class of Calabi--Yau varieties.

In dimension two, klt Calabi--Yau surfaces (also known as log Enriques surfaces) with worse than du Val singularities are rationally connected Calabi--Yau varieties (see \cite[Proof of Lemma 1.4]{AM} or \cite[Theorem D(1)]{Blache}) and form a bounded family by a result of Alexeev \cite[Corollary 6.10]{AK2}. The works of Blache and Zhang \cite{Blache, DQ1, DQ2} provide a systematic study of such surfaces together with many examples.
Interesting examples of rationally connected klt Calabi--Yau $3$-fold can also be found in \cite{OT} where the authors show that $E^3/G$ is a rational klt Calabi--Yau $3$-fold. Here $E$ is the elliptic curve corresponding to the regular hexagonal lattice in $\mathbb{C}$ and $G$ is the group generated by an automorphism corresponding to multiplication by a primitive third root of unity on the lattice, see \cite[Remark 2.8]{OT} for details. For more examples, we refer to \cite{Cam, COT, CT, COV}. Unfortunately, there are not many known examples of rationally connected klt Calabi--Yau varieties of dimension $d\geq 3$. 
We expect that one can construct more examples by considering finite quotients of smooth Calabi--Yau varieties (cf. \cite{KL}). In the Appendix \ref{appendix}, we provide a sufficient condition for a dlt log Calabi--Yau pair to be rationally connected.


We note that rationally connected Calabi--Yau varieties appear as bases of elliptic Calabi--Yau manifolds. For example, in \cite{Ogu93}, Oguiso proved that the base of an elliptic Calabi--Yau $3$-fold $X$ is a rationally connected Calabi--Yau surface when a semiample divisor $D$ defining the fibration satisfies the extra condition $c_2(X)\cdot D =0$. Oguiso provided explicit examples of such fibrations. In general, the structure theorem proved in \cite[Theorem 3.2]{DCS} shows that they appear as the base of an elliptic Calabi-Yau manifold in a rather special situation. Roughly speaking, if the base is a rationally connected Calabi--Yau variety then the fibration is isotrivial, up to a birational modification, and the total space behaves like a product, or more precisely it is of product type. See~\cite{DCS, BDCS} for more details.  

Motivated by Alexeev's work in dimension two, we may expect that boundedness holds for rationally connected klt Calabi--Yau varieties and we are lead to consider the following conjecture.
\begin{conj}\label{conj1}
	Fix a positive integer $d$. The set of all rationally connected klt Calabi--Yau varieties of dimension $d$ forms a bounded family.
\end{conj}

As it is not hard to show that the singularities of rationally
klt Calabi--Yau varieties have bounded discrepancies (cf. Lemma \ref{klt=e-klt}), 
Conjecture \ref{conj1} is a special case of the following conjecture generalizing the BAB Conjecture.

\begin{conj}[{cf. \cite{AK2}, \cite[Conjecture 3.9]{MP04}}]\label{conj2}
	Fix a positive real number $\epsilon$ and a positive integer $d$. The set of all $X$ satisfying
	\begin{enumerate}
		\item $\dim X=d,$
		\item there exists a boundary $B$ such that $(X, B)$ is $\epsilon$-klt,
		\item $-(K_X+B)$ is nef, and
		\item $X$ is rationally connected,
	\end{enumerate}
	forms a bounded family.
\end{conj}
In dimension two, Conjecture \ref{conj2} was proved by Alexeev \cite[Theorem 6.9]{AK2}. 
However, it still remains open even in dimension three. 
Recently, Birkar together with the second and fifth named authors  proved that Conjecture~\ref{conj2} holds in dimension 3 up to isomorphism in codimension one \cite[Theorem~1.6]{BDCS}. 
The conjecture remains wide open in higher dimension.

Conjecture \ref{conj2} is already very interesting in the case when $K_X+B\equiv 0$ which can be formulated separately as the following statement. 
\begin{conj}\label{conj3}
	Fix a positive real number $\epsilon$ and a positive integer $d$. The set of rationally connected $d$-dimensional varieties of $\epsilon$-CY type forms a bounded family.
\end{conj}
Here, a normal projective variety $X$ is \emph{of $\epsilon$-CY type} if there exists an effective $\bR$-divisor $B$ such that $(X, B)$ is an $\epsilon$-klt log Calabi--Yau pair. 

In \cite[Theorem~1.4]{BDCS}, recently an affermative answer was given showing that the birational boundedness version of for Conjecture~\ref{conj3} holds in the case when the torsion index of $K_X+B\equiv 0$ is bounded. 
However, to conclude Conjecture \ref{conj1} from Conjecture \ref{conj2}, it is necessary to consider the case when the coefficients of $B$ do not belong to a fixed DCC set or $B= 0$, as the minimal log discrepancy is not known to satisfy the ACC property.

The goal of this article is to study Conjecture \ref{conj3} in dimension three and to establish several birational boundedness results. 

The first result provides an affirmative answer 
to the birational boundedness in Conjecture \ref{conj1} for dimension three.
\begin{thm}\label{main}
	The set of all rationally connected klt Calabi--Yau
	$3$-folds forms a birationally bounded family.
\end{thm}
Theorem \ref{main} is a special case of the following theorem, which gives an affirmative answer to the birational boundedness for Conjecture \ref{conj3} in dimension three.
\begin{thm}\label{main2}
	Fix a positive real number $\epsilon$. The set of rationally connected $3$-folds of $\epsilon$-CY type forms a birationally bounded family.
\end{thm}
Theorem \ref{main2} can be viewed as a generalization of \cite[Theorem 1.3]{DCS} in dimension $3$, and also a generalization of the birational BAB conjecture in dimension $3$ \cite{Jiang6}.

We moreover focus on $\epsilon$-lc log Calabi--Yau pairs $(X, B)$ such that the coefficients of $B$ are bounded from below. For such pairs, we show that log boundedness modulo flops holds: this is a stronger version of (log) birational boundedness, see Section \ref{sec.bdd} for the definition.

\begin{thm}[{=Corollary \ref{cor.bdd.lcy3fold}}]
	Fix positive real numbers $\epsilon$, $\delta$. 
	Then, the set of pairs $(X,B)$ satisfying
	\begin{enumerate}
		\item $(X,B)$ is an $\epsilon$-lc log Calabi--Yau 
		pair of dimension $3$,
		\item $X$ is rationally connected, and
		\item $B>0$, and the coefficients of $B$ are at least $\delta$,
	\end{enumerate}
	forms a log bounded family modulo flops. 
\end{thm}
In fact, this result can be generalized to any dimension modulo Conjecture \ref{conj3} in lower dimensions, see Theorem \ref{t.bdd.lcy3fold}. It is a consequence of a relative version of the Special BAB \cite[Theorem 1.4]{Bir16a} (see Theorem \ref{t.bdd.mfs}).

We apply this result to show that rationally connected klt Calabi--Yau $3$-folds with mld bounded away from $1$ are bounded modulo flops.
\begin{thm}[{=Theorem \ref{t.bdd.mld}}]\label{thm: bdd modulo flops 3-dim CY RC}
	Fix $0<c<1$.
	Let $\mathcal{D}$ be the set of varieties $X$ such that
	\begin{enumerate}
		\item $X$ is a rationally connected Calabi--Yau $3$-fold, and
		\item $0<\mld(X)< c$.
	\end{enumerate}
	Then $\mathcal{D}$ is bounded modulo flops. 
\end{thm}

This theorem has several interesting immediate applications to the boundedness problem. We show that for rationally connected klt Calabi--Yau $3$-folds, boundedness modulo flops is equivalent to the boundedness of global indices (Corollary \ref{c.bdd.index})
and that boundedness modulo flops holds modulo 1-Gap conjecture for minimal log discrepancies on $3$-folds, which is a special case of Shokurov's ACC conjecture (Corollary \ref{t.bdd.cyklt}). Finally, we establish that the boundedness modulo flops holds for those rationally connected klt Calabi--Yau varieties which are quasi-\'etale quotients of irregular
varieties (Corollary \ref{t.bdd.ab.cy}). We refer the reader to Section \ref{section 7} for details.

\medskip

\noindent\textbf{Postscriptum}. After our paper was first posted on the arXiv, the fourth named author proved the 1-Gap conjecture, Conjecture~\ref{c.mld}, for minimal log discrepancies on $3$-folds \cite[Theorem 1.3]{Jiang-gap}, see also \cite[Theorem 1.4]{JLX19}. Thus the assumption ``$0<\mld(X)<c$'' in Theorem \ref{thm: bdd modulo flops 3-dim CY RC} could be replaced by ``$0<\mld(X)<1$'', and Corollary \ref{t.bdd.cyklt} holds unconditionally, see \cite[Theorem 1.6]{Jiang-gap}.

\medskip

\noindent\textbf{Acknowledgments}.
The authors would like to thank Yoshinori Gongyo for suggesting this topic.
WC was supported by NAKAMURA scholarship and the UTokyo System of Support for Graduate Research.
WC would like to thank his advisor Yoshinori Gongyo for supporting a visit to the University of Cambridge, where he had valuable discussions with Caucher Birkar and RS. 
DC was supported in part by NSF Grant DMS-1702358.
JH would like to thank his advisors Gang Tian and Chenyang Xu in particular for constant support and encouragement. 
CJ was supported by JSPS KAKENHI Grant Number JP16K17558 and World Premier International Research Center Initiative (WPI), MEXT, Japan. 
RS was partially supported by Churchill College, Cambridge.
Part of this work was completed during a visit of RS to Princeton University. 
RS would like to thank Princeton University for its hospitality and the nice working environment, and J\'anos Koll\'ar for funding his visit.
We are grateful to Keiji Oguiso for discussions on examples, and Zhiyu Tian for discussions related to the material in the Appendix. 
We thank the referees for useful comments.

\section{Preliminaries}

We adopt the standard notation and definitions in \cite{KMM} 
and \cite{KM}, and we will freely use them.

\subsection{Pairs, singularities, and mld}
A {\it log pair} $(X, B)$ consists of a normal projective variety $X$ 
and an effective $\bR$-divisor $B$ on $X$ such that $K_X+B$ is $\bR$-Cartier.

Let $f\colon Y\rightarrow X$ be a log
resolution of the log pair $(X, B)$, write
\[
K_Y =f^*(K_X+B)+\sum a_iF_i,
\]
where $\{F_i\}$ are distinct prime divisors.  
For a non-negative real number $\epsilon$, the log pair $(X,B)$ is called
\begin{itemize}
	\item[(a)] \emph{$\epsilon$-kawamata log terminal} (\emph{$\epsilon$-klt},
	for short) if $a_i> -1+\epsilon$ for all $i$;
	\item[(b)] \emph{$\epsilon$-log canonical} (\emph{$\epsilon$-lc}, for
	short) if $a_i\geq  -1+\epsilon$ for all $i$;
	\item[(c)] \emph{terminal} if  $a_i> 0$ for all $f$-exceptional divisors $F_i$ and all $f$;
	\item[(d)] \emph{canonical} if  $a_i\geq  0$ for all $f$-exceptional divisors $F_i$ and all $f$.
\end{itemize}

Usually we write $X$ instead of $(X,0)$ in the case $B=0$.
Note that $0$-klt (resp., $0$-lc) is just klt (resp., lc) in the usual sense. Also note that 
$\epsilon$-lc singularities only make sense if $\epsilon\in [0,1]$, and  $\epsilon$-klt 
singularities only make sense if $\epsilon\in [0,1)$,

The {\it log discrepancy} of the divisor $F_i$ is defined to be $a(F_i, X, B)=1+a_i$.
It does not depend on the choice of the log resolution $f$.
$F_i$ is called a {\it non-lc place} of $(X, B)$  if $a_i< -1$.
A subvariety $V\subset X$ is called a {\it non-lc center} of 
$(X, B)$ if it is the image of a non-lc place. 
The {\it non-lc locus} $\text{Nlc}(X, B)$ is the union of 
all non-lc centers of $(X, B)$.

Let $(X, B)$ be an lc pair and $Z \subset X$ an irreducible closed subset with $\eta_Z$ the generic point of $Z$. 
The {\it minimal log discrepancy} of $(X, B)$ over $Z$ is defined as
\[
\mathrm{mld}_Z(X, B)=\inf\{a(E,X,B)\mid \text{center}_{X}(E)\subset Z\},\]
and the minimal log discrepancy of $(X, B)$ at $\eta_Z$ is defined as
\[
\mathrm{mld}_{\eta_Z}(X, B)=\inf\{a(E,X,B)\mid \text{center}_{X}(E)= Z\}.\]
For simplicity, we just write $\mathrm{mld}(X, B)$ instead of $\mathrm{mld}_X(X, B)$.

\subsection{Log Calabi--Yau pairs}
The log pair $(X, B)$ is called a
\emph{log Calabi--Yau pair} if $K_X+B\equiv 0$. 
Recall that if $(X, B)$ is lc, this is equivalent 
to $K_X+B\sim_\bR 0$ by \cite{G}.

A normal projective variety $X$ is \emph{of $\epsilon$-CY type} 
if there exists an effective $\bR$-divisor $B$ such that $(X, B)$ 
is an $\epsilon$-klt log Calabi--Yau pair. 

\subsection{Terminal Mori fibrations}
A projective morphism $f\colon X\to Z$ between normal 
projective varieties is called a {\it terminal Mori fibration} 
(or {\it terminal Mori fiber space}) if
\begin{enumerate}
	\item $X$ is $\bQ$-factorial with terminal singularities;
	\item $f$ is a {\it contraction}, i.e., $f_*\OO_X=\OO_Z$;
	\item $-K_X$ is ample over $Z$;
	\item $\rho(X/Z)=1$;
	\item $\dim X > \dim Z$.
\end{enumerate}

We say that $X$ is endowed with a 
{\it terminal Mori fibration structure} if there 
exists a terminal Mori fibration $X\to Z$. 
In particular, in this situation, $X$ has at most
$\bQ$-factorial terminal singularities by definition.

\subsection{Bounded pairs}\label{sec.bdd}
A collection of varieties $ \mathcal{D}$ is
said to be \emph{bounded} (resp., 
\emph{birationally bounded}, or \emph{bounded in codimension one})
if there exists 
$h\colon \mathcal{Z}\rightarrow S$ a projective morphism 
of schemes of finite type such that
each $X\in \mathcal{D}$ is isomorphic (resp., birational, 
or isomorphic in codimension one) to $\mathcal{Z}_s$ 
for some closed point $s\in S$.

We say that a collection of log pairs $\mathcal{D}$ is 
{\it log birationally bounded} (resp.,  \emph{log bounded}, 
or \emph{log bounded in codimension one})
if there is a  quasi-projective scheme $\mathcal{Z}$, a 
reduced divisor $\mathcal{E}$ on $\mathcal Z$, and a 
projective morphism $h\colon \mathcal{Z}\to S$, where 
$S$ is of finite type and $\mathcal{E}$ does not contain 
any fiber, such that for every $(X,B)\in \mathcal{D}$, 
there is a closed point $s \in S$ and a birational
map $f \colon \mathcal{Z}_s \dashrightarrow X$ 
(resp., isomorphic, or isomorphic in codimension one)
such that $\mathcal{E}_s$ contains the support of $f_*^{-1}B$ 
and any $f$-exceptional divisor (resp., $\mathcal{E}_s$ 
coincides with the support of $f_*^{-1}B$, $\mathcal{E}_s$ 
coincides with the support of $f_*^{-1}B$).

Moreover, if $\mathcal{D}$ is a set of klt Calabi--Yau 
varieties (resp., klt log Calabi--Yau pairs), then it is 
said to be {\it bounded modulo flops} (resp., {\it log 
	bounded modulo flops}) if it is (log) bounded in 
codimension one, and each fiber $\mathcal{Z}_{s}$ 
corresponding to $X$ in the definition is normal projective, 
and $K_{\mathcal{Z}_s}$ is $\bQ$-Cartier (resp., $K_{\mathcal{Z}_s}+f_*^{-1}B$ is $\bR$-Cartier). 

Note that if $\mathcal{D}$ is a set of klt log Calabi--Yau 
pairs which is log bounded modulo flops, and 
$(X, B)\in \mathcal{D}$ with a birational
map $f \colon \mathcal{Z}_s \dashrightarrow X$ that is an isomorphism in codimension one as in the definition, then $(\mathcal{Z}_s, f_*^{-1}B)$ is again a klt log 
Calabi--Yau pair by the Negativity Lemma. 
Moreover,  $(X, B)$ is $\epsilon$-lc if and only 
if $(\mathcal{Z}_s, f_*^{-1}B)$ is so. A similar statement 
holds for $\mathcal{D}$ a set of klt Calabi--Yau varieties.

Here the name ``modulo flops" comes from the fact that, if we assume that $X$ and $\mathcal{Z}_s$ are both $\bQ$-factorial, then they are connected by flops by running a $(K_X+B+\delta f_*H)$-MMP where $H$ is an ample divisor on $\mathcal{Z}_s$ and $\delta$ is a sufficiently small positive number  (cf. \cite{BCHM, flops}) .

\subsection{Volume}
Let $X$ be a $d$-dimensional projective variety  and $D$ 
a Cartier divisor on $X$. The {\it volume} of $D$ is the real number
\[
{\Vol}(X, D)=\limsup_{m\rightarrow \infty}\frac{h^0(X,\OO_X(mD))}{m^d/d!}.
\]
For more background on the volume, see \cite[2.2.C]{Positivity1}. 
By homogeneity and continuity of the volume, we 
can extend the definition to $\bR$-Cartier $\bR$-divisors. 
Moreover, if $D$ is a nef $\bR$-divisor, then $\Vol(X, D)=D^d$.
If $D$ is a $\bR$-divisor which is not $\bR$-Cartier, if a $\bQ$-factorialization of $X$, i.e., a birational morphism $\phi \colon Y\to X$ which is an isomorphism in codimension one and $Y$ is $\bQ$-factorial, exists then we define $\Vol(X, D):=\Vol(Y, \phi^{-1}_*D)$. 
Note that $\bQ$-factorializations always exist for varieties $X$ which admit an effective $\bR$-divisor $B$ such that $(X,B)$ is klt (cf. \cite[Corollary 1.4.3]{BCHM}). 

It is easy to see the following inequality for volumes 
by comparing global sections by exact sequences.

\begin{lem}[{\cite[Lemma 2.5]{Jiang6}, \cite[Lemma 4.2]{DCS}}]\label{lemma volume}
	Let $X$ be a projective normal variety, $D$ an $\bR$-Cartier
	$\bR$-divisor, and $S$ a base-point free normal Cartier prime divisor. 
	Then for any real number $q>0$,
	\[
	\Vol(X,D+qS)\leq \Vol(X, D)+q(\dim X) \Vol(S, D|_S+qS|_S).
	\]
\end{lem}

\subsection{Length of extremal rays}
Recall the following result on the length of extremal rays due to Kawamata.
\begin{thm}[{\cite{Ka}}]\label{ext ray}
	Let $(X, B)$ be a klt pair. Then every $(K_X+B)$--negative extremal 
	ray $R$ is generated by the class of a rational curve $C$ such that 
	\[
	0<-(K_X+B)\cdot C\leq 2\dim X.
	\]
\end{thm}
However, as we need to deal with non-klt pairs in the applications, we will use the following  
generalization of this theorem for log pairs which was proved by Fujino, cf.~\cite{S19}.

\begin{thm}[{\cite[Theorem 1.1(5)]{Fujino}}]\label{ext ray2}
	Let $(X, B)$ be a log pair. Let $i \colon 
	\Nlc(X,B) \to X$ be the inclusion of the non-lc locus in $X$.
	Fix a $(K_X+B)$-negative extremal ray $R$. Assume that
	\[
	R\cap \overline{NE}(X)_{\Nlc(X,B)} = \{0\},
	\]
	where
	\[
	\overline{NE}(X)_{\Nlc(X,B)} = {\rm Im}(i_\ast \colon 
	\overline{NE}(\Nlc(X,B)) \to \overline{NE}(X)).
	\]
	Then $R$ is generated by a rational curve $C$ such that 
	\[
	0<-(K_X+B)\cdot C\leq 2\dim X.
	\]
\end{thm}

\section{Birational boundedness of rationally 
	connected Calabi--Yau $3$-folds}\label{sec 3}

The goal of this section is to prove Theorems \ref{main} and \ref{main2}.
\subsection{Sketch of the proof}
The strategy of proof for Theorems \ref{main} and \ref{main2} originates from \cite{Jiang6}.
Using the minimal model program, it suffices to work with varieties of CY-type 
endowed with a terminal Mori fibration structure. The precise result, 
whose proof will be given in Section \ref{subsection MFS}, is the following. 

\begin{prop}[{cf. \cite[Proof of Theorem 2.3]{Jiang}}]\label{thm MFS}
	Fix a positive real number $\epsilon$ and a positive integer $d$.
	Every $d$-dimensional rationally connected variety $X$ of 
	$\epsilon$-CY type is birational to a $d$-dimensional rationally 
	connected variety $X'$ of $\epsilon$-CY type with a terminal Mori fibration structure. 
\end{prop}

Now let $X$ be a $3$-fold of $\epsilon$-CY type with a terminal Mori 
fibration $f\colon X\to Z$. If $-K_X$ is big, then $X$ is of Fano type, 
and the (birational) boundedness follows from the BAB conjecture 
in dimension $3$, see \cite[Corollary 1.8]{Jiang6} or \cite[Corollary 1.2]{Bir16b}. 
Thus, we only need to consider the case when $-K_X$ is not big. 
Since $-K_X$ is ample over $Z$, this implies that $\dim Z>0$. 
In this case, we prove the following theorem.

\begin{thm}\label{NV}
	Fix a positive real number $\epsilon$. 
	Then there exist positive integers $n=n({\epsilon}), c=c(\epsilon)$, and 
	$v=v({\epsilon})$ depending only on $\epsilon$, with the following property:
	
	Assume $X$ is a rationally connected $3$-fold of $\epsilon$-CY type
	endowed with a terminal Mori fibration $f\colon X\to Z$ such that $-K_X$ is not big.
	\begin{enumerate}
		\item If $\dim Z=1$ (i.e. $Z=\bP^1$), take a general fiber $F$ of $f$, then
		
		\begin{enumerate}[label=(1.\arabic*)]
			\item   $-K_X+n F$ is ample,
			\item  $|-3K_X+8n F|$ defines a birational map, and
			\item $(-K_X+n F)^3\leq v$.
		\end{enumerate}
		
		\item If $\dim Z=2$, then there exists a very ample divisor $H$ on $Z$ such that
		
		\begin{enumerate}[label=(2.\arabic*)]
			\item $H^2 \leq c$,
			\item  $-K_X+n f^*H$ is ample,
			\item  $|-3K_X+8n f^*H|$ defines a birational map, and
			\item $(-K_X+n  f^*H)^3\leq v$.
		\end{enumerate}
	\end{enumerate}
\end{thm}

The proof of Theorem \ref{NV} will be given in 
Section \ref{section proof 2}, while the proof of 
Theorems \ref{main} and \ref{main2} will be 
given in Section \ref{subsec mains}.

\subsection{Proof of Proposition \ref{thm MFS}}\label{section proof 1}

In this subsection, for the reader's convenience, 
we recall the proof of Proposition \ref{thm MFS}.

\begin{lem}\label{RC not pseff}
	If $X$ is rationally connected and with at worst canonical singularities, then $K_X$ is not pseudo-effective.
\end{lem}
\begin{proof}
	Take a resolution $\phi: Y \to X$, $Y$ is again rationally connected, 
	hence $K_Y$ is not pseudo-effective. Since $X$ is canonical, $K_Y\geq \phi^*K_X$, and therefore $K_X$ is not pseudo-effective.
\end{proof}

\begin{proof}[Proof of Proposition \ref{thm MFS}]\label{subsection MFS}
	Fix a positive real number $\epsilon$ and a positive integer $d$. 
	Let  $X$ be a rationally connected variety of $\epsilon$-CY type 
	of dimension $d$.
	By \cite[Corollary 1.4.3]{BCHM}, taking a terminalization of $(X, B)$, 
	we have a birational morphism $\pi\colon X_1 \rightarrow X$ where
	$K_{X_1}+B_1=\pi^*(K_{X}+B), \;B_1> 0$ is an effective $\bR$-divisor, 
	and $X_1$ is $\bQ$-factorial terminal.
	Here $K_{X_1}+B_1\equiv 0$ and $(X_1,B_1)$ is  
	$\epsilon$-klt; moreover, $X_1$ is again 
	rationally connected. In particular, $K_{X_1}$ 
	is not pseudo-effective by Lemma \ref{RC not pseff} since $X_1$ is terminal.
	
	We can run a $K_{X_1}$-MMP with scaling of an ample 
	divisor on $X_1$, 
	which terminates with a Mori fiber space $X'\rightarrow T$, 
	cf. \cite[Corollary 1.3.3]{BCHM}. 
	As we run a $K_{X_1}$-MMP, $X'$ is again $\bQ$-factorial terminal 
	and rationally connected.
	By the Negativity Lemma, $K_{X'}+B'\equiv 0$ and 
	$(X',B')$ is  $\epsilon$-klt where $B'$ is the strict transform 
	of $B_1$ on $X'$. Now $X'$ is a $d$-dimensional rationally 
	connected variety of $\epsilon$-CY type with a terminal Mori 
	fibration structure by construction, which is birational to $X$.
	This concludes the proof.
\end{proof}

\subsection{Proof of Theorem \ref{NV}}\label{section proof 2}
In this subsection, we prove Theorem \ref{NV}. This will follow directly 
from Lemmas \ref{lemma ample}, \ref{lemma birationality}, 
and \ref{lem volume} below.

\subsubsection{Setting}\label{setting}
Fix a positive real number $\epsilon$. Let $X$ be a rationally connected 
$3$-fold of $\epsilon$-CY type with a terminal Mori 
fibration $f \colon X\to Z$ such that $-K_X$ 
is not big and $\dim Z>0$. Suppose $(X, B)$ is an 
$\epsilon$-klt log Calabi--Yau pair. We define a base-point 
free divisor $G$ on $X$, coming from the boundedness 
of the base $Z$, in the following way.

When $\dim Z=1$, then $Z=\bP^1$. 
In this case, $G$ is defined to be a general fiber of $f$, 
which is a smooth del Pezzo surface since $X$ is terminal. 
$(G, B|_G)$ is an  $\epsilon$-klt log Calabi--Yau pair by 
adjunction.

If $\dim Z=2$, then the collection of such $Z$ forms a bounded family.
In fact, since $X$ is of $\epsilon$-CY type, there exists an effective 
$\bR$-divisor $\Delta$ such that $(Z, \Delta)$ is $\delta$-klt, 
$K_Z+\Delta\sim_\bR 0$, see \cite[Corollary 1.7]{Birkar}.
Here $\delta=\delta(\epsilon)$ is a positive number depending only on $\epsilon$. 
Hence $Z$ is rationally connected and of $\delta$-CY type. 
The boundedness then follows from the solution to the BAB Conjecture 
in dimension  $2$ (see \cite[Theorem 6.9]{AK2} or \cite[Lemma 1.4]{AM})
This implies that there is a positive integer $c=c(\epsilon)$ 
depending only on $\epsilon$ and we can find a general 
very ample divisor $H$ on $Z$ satisfying $H^2\leq c$. 
We take $G=f^*H$, then $G$ is a conic bundle over the curve $H$ 
(i.e. $-K_G$ is relatively ample over $H$). 
Note that $H$ and $G$ are smooth since $H$ is general and $X$ is terminal.
Also $(G, B|_G)$ is $\epsilon$-klt and $-(K_G+B|_G)+G|_G\sim_\bR 0$ by adjunction. 
Moreover, $G|_G=f|_G^*(H|_H)\equiv(H^2)F$ where $F$ is a general fiber of $f|_G$.
Finally, since $\rho(X/Z)=1$, $B^v\sim_{\bR, f} 0$, where $B^v$ is the 
$f$-vertical part of $B$, and hence $B^v|_G\sim_{\bR,f|_G} 0$.

\subsubsection{A boundedness theorem on surfaces}
We recall the following boundedness theorem for surfaces from \cite{Jiang6}. 
The ideas behind its proof are inspired by the solution to the BAB Conjecture 
in dimension two given by Alexeev--Mori \cite{AM}.

\begin{thm}[{\cite[Theorem 5.1]{Jiang6}}]\label{thm surface}
	Fix a positive integer $m$ and a positive real number $\epsilon$. 
	Then there exists a number $\lambda'=\lambda'(m,\epsilon)>0$ 
	depending only  on $m$ and $\epsilon$ satisfying the following property:
	
	Assume that $T$ is a projective smooth surface and $B=\sum_ib_iB^i$ 
	an effective $\bR$-divisor on $T$ where each $B^i$ is a prime divisor such that
	\begin{enumerate}
		\item $(T, B)$ is $\epsilon$-klt, but $(T, (1+t)B)$ is not klt for some $t>0$,
		\item $K_T+B\equiv N-A$ where $A$ is an ample $\bR$-divisor 
		and $N$ is a nef $\bR$-divisor on $T$,
		\item  $\sum_ib_i\leq m$,
		\item $B^2\leq m$, $(B\cdot N)\leq m$.
	\end{enumerate}
	Then $t>\lambda'$.
\end{thm}

For the proof, we refer to {\cite[Theorem 5.1]{Jiang6}}. Note that {\cite[Theorem 5.1]{Jiang6}} only treats $\bQ$-divisors, but the same proof applies to $\bR$-divisors.
By applying Theorem \ref{thm surface} to our situation, 
we can show the following theorem, which is a simple 
modification of \cite[Theorem 1.7]{Jiang6}.

\begin{thm}[{cf. \cite[Theorem 1.7]{Jiang6}}]\label{conj LCT}
	Fix a positive real number $\epsilon$. 
	Then there exists a number  $\lambda=\lambda(\epsilon)>0$ 
	depending only on $\epsilon$, satisfying the following property:
	
	\begin{enumerate}
		\item If $(G, B)$ is an $\epsilon$-klt log Calabi--Yau pair 
		and $G$ is a smooth del Pezzo surface, then 
		$(G, (1+t)B)$ is klt for $0<t\leq \lambda$.
		
		\item If $f\colon G\to H$ is a conic bundle from a smooth surface 
		$G$ to a smooth curve, $(G, B)$ is an $\epsilon$-klt pair, 
		$-(K_G+B)+k F\equiv 0$ for some integer $k\leq c$, and 
		$B^{v}\sim_{\bR,f} 0$, then $(G, (1+t)B)$ is klt for $0<t\leq \lambda$. 
		Here $F$ is a general fiber of $f$, $B^v$ is the $f$-vertical part of $B$, 
		and $c$ is the number depending only on $\epsilon$ defined in 
		Section \ref{setting}.
	\end{enumerate}
\end{thm}

\begin{proof}
	(1) As $G$ is a del Pezzo surface, it follows that $-K_G$ is ample, 
	$-3K_G$ is very ample, and $(-K_G)^2\leq 9$.
	Write $B=\sum_ib_iB^i$, then
	\begin{align*}
	\sum_i b_i\leq{}& B\cdot (-K_G)= (-K_G)^2\leq 9;\\
	B^2={}& (-K_G)^2\leq 9.
	\end{align*}
	As $(G, B)$ is $\epsilon$-klt, we can apply Theorem \ref{thm surface} 
	for $A=N=-K_G$ and obtain that $(G, (1+t)B)$ is klt for all $0<t\leq \lambda'(9, \epsilon)$.
	
	(2) Suppose that $f:G\to H$ is a conic bundle from a 
	smooth surface $G$ to a smooth curve, $(G, B)$ is an 
	$\epsilon$-klt pair, $-(K_G+B)+k F\equiv 0$ for some 
	integer $k\leq c$, and $B^{v}\sim_{\bR,f} 0$. 
	The assumption $B^{v}\sim_{\bR,f} 0$ implies that 
	we may write $B=\sum_ib_iB^i+\sum_jc_jF^j$, where 
	$B^i$ is a curve not contained in a fiber of $f$ for all $i$, 
	and $F^j$ is a fiber of $f$ for any $j$. This condition is 
	crucial in the following claim. Note that each $F^j$ is reduced 
	and contains at most 2 irreducible components since 
	$f$ is a conic bundle.
	Moreover, recall that $B\cdot F= (-K_G)\cdot F=2$ and 
	$(-K_G)^2\leq 8$ for the conic bundle $G$.
	
	\begin{claim}\label{claim bi}
		The sum of coefficients of $B$ is bounded from above by 
		\[
		\sum_ib_i+\sum_j 2c_j\leq 8+2k\leq 8+2c.
		\]
	\end{claim}
	\begin{proof}[Proof of Claim \ref{claim bi}]
		First of all, we have
		\[
		\sum_ib_i\leq \sum_ib_i(B^i\cdot F)=(B\cdot F)= 2.
		\]
		Hence,
		it suffices to show that $\sum_jc_j\leq  3+k$. 
		Assume, to the contrary,  that $w=\sum_jc_j>3+k$. 
		Then, for any choice of three sufficiently general fibers 
		$F_1, \; F_2, \; F_3$ of $f$, consider the pair
		\[
		K_G+\sum_ib_iB^i+ \Big(1-\frac{3+k}{w}\Big)\sum_jc_jF^j+F_1+F_2+F_3\sim_\bR 0.
		\]
		Applying \cite[Theorem 4.37]{Kol13} to $X=G$, $Z$ a point, and $D=F_1+F_2+F_3$, 
		we conclude that $D$ has 2 connected components, which is obviously absurd.
	\end{proof}
	
	Moreover, we have
	\begin{align*}
	B^2= {}&(kF-K_G)^2=4k+(-K_G)^2\leq 4c+8;\\
	(B\cdot kF)= {}& (-K_G)\cdot kF =2k\leq 2c.
	\end{align*}
	
	Applying Theorem \ref{thm surface} for $N=kF+A$, where $A$ is a sufficiently small ample 
	$\bQ$-divisor such that $(A\cdot B)\leq c$, and fixing $m=4c+8$, we obtain that 
	$(G, (1+t)B)$ is klt for all $0<t\leq \lambda'(4c+8, \epsilon)$.
\end{proof}
We propose the following conjecture generalizing 
Theorem \ref{conj LCT} to higher dimension. 

\begin{conj}[{cf. \cite[Conjecture 1.13]{Jiang6}}]\label{conj LCT higher}
	Fix a positive real number $\epsilon$ and a positive integer $d$. 
	There exists a positive number $t=t(d,\epsilon)$ 
	depending only on $d$ and $\epsilon$, such that for any $d$-dimensional 
	$\epsilon$-klt log Calabi--Yau pair $(X,B)$, $(X,(1+t)B)$ is klt.
\end{conj}

\subsubsection{Effective construction of an ample divisor}
\begin{lem}\label{lemma ample}
	Under the setting introduced in Section \ref{setting},
	there exists a positive integer $n=n(\epsilon)$ depending 
	only on $\epsilon$ such that $-K_X+kG$ is ample for all $k\geq n$.
\end{lem}
\begin{proof}
	By construction, $(G, B|_G)$ satisfies one of the two conditions in Theorem \ref{conj LCT}.
	Hence $(G, (1+\lambda)B|_G)$ is klt for $\lambda>0$ and $\lambda$ depends only on
	$\epsilon$ by Theorem \ref{conj LCT}.
	
	Therefore, in either case, every curve in $\Nlc(X, (1+ \lambda)B)$ is contracted by $f$, 
	by inversion of adjunction. That means that $f(\Nlc(X, (1+ \lambda)B))$ is a finite set of (closed) points. 
	In particular, every curve $C_0$ supported in $\Nlc(X, (1+ \lambda)B)$ 
	satisfies the equality $G\cdot C_0 =0$, since $G$ is the pull-back of 
	an ample divisor on $Z$. This implies that every class 
	$C\in \overline{NE}(X)_{\Nlc(X, (1+ \lambda)B)}$ satisfies $G\cdot C=0$.
	
	Let us consider an extremal ray $R$ of $\overline{NE}(X)$. Note that $G\cdot R\geq 0$.
	
	If $G\cdot R= 0$,  then $R$ is contracted by $f$ since $G$ is the pull-back of an ample
	divisor on $Z$ and $-K_X\cdot R>0$, as $-K_X$ is ample over $Z$.
	
	If $G\cdot R> 0$ and $R$ is $(K_X+(1+\lambda)B)$-non-negative, 
	then
	\[
	\Big(-K_X+\frac{7}{\lambda}G\Big)\cdot R=
	\frac{1}{\lambda}(K_X+(1+\lambda)B))\cdot R+\frac{7}{\lambda}G\cdot R> 0,
	\]
	as $K_X+B\equiv 0$.
	
	If $G\cdot R>0$ and $R$ is $(K_X+(1+ \lambda)B)$-negative, then
	\[
	R\cap \overline{NE}(X)_{\Nlc(X, (1+ \lambda)B)} = \{0\},
	\]
	since we showed that $G\cdot C =0$ for any class 
	$C\in \overline{NE}(X)_{\Nlc(X, (1+ \lambda)B)}$.
	By Theorem \ref{ext ray2}, $R$ is generated by a rational curve $C'$ such that
	\[
	(K_X+(1+ \lambda)B)\cdot C' \geq -6.
	\]
	On the other hand, $G\cdot C'\geq 1$ since $G\cdot C'>0$ and $G$ is Cartier.
	Hence,
	\begin{align*}
	{}&\Big(-K_X+\frac{7}{\lambda}G\Big)\cdot C'\\
	={}&\frac{1}{\lambda}(K_X+(1+\lambda)B)\cdot C'+\frac{7}{\lambda}G\cdot C'> 0.
	\end{align*}
	In summary, the inequality
	\[
	(-K_X+kG)\cdot R> 0
	\]
	holds for any extremal ray $R$ and for any $k\geq \frac{7}{\lambda}$,
	as $G$ is nef. By Kleiman's Ampleness Criterion,
	$-K_X+kG$ is ample for all $k\geq \frac{7}{\lambda}$. We may take 
	$n=\roundup{7/\lambda}$ to complete the proof.
\end{proof}

\begin{remark}
	If Conjecture \ref{conj LCT higher} holds in dimension $3$, then 
	Lemma \ref{lemma ample} is an easy consequence 
	of Kawamata's estimates on the length of extremal rays. 
	However, the conjecture is still wide open, so we need to use the 
	result for surfaces contained in Theorem \ref{thm surface}.
\end{remark}

\subsubsection{Boundedness of birationality}
\begin{lem}\label{lemma birationality}
	Under the setting of Section \ref{setting}, the linear system 
	$|-3K_X+kG|$ defines a birational map for $k\geq 4n+4$, where $n$ is the natural number 
	given in Lemma \ref{lemma ample}.
\end{lem}
\begin{proof}
	By Lemma \ref{lemma ample}, $-K_X+nG$ is ample.
	
	If $\dim Z=1$, then $G$ is a  smooth del Pezzo surface. 
	It is well-known that $|-3K_G|$ gives a birational map (in fact, an embedding).
	For two general fibers $G_1$ and $G_2$ of $f$, and for an integer $k\geq 4n+2$, 
	consider the short exact sequence
	\begin{align*}
	0\to {}&\OO_X(-3K_X+kG-G_1-G_2)\to \OO_X(-3K_X+kG)\\
	\to{}& \OO_{G_1}(-3K_{G_1})\oplus   \OO_{G_2}(-3K_{G_2})\to 0.
	\end{align*}
	Since $k\geq 4n+2$,
	$-4K_X+kG-G_1-G_2$
	is ample, by the Kawamata--Viehweg vanishing theorem
	\begin{align*}
	{}&H^1(X, \OO_X(-3K_X+kG-G_1-G_2))\\
	={}&H^1(X, \OO_X(K_X-4K_X+kG-G_1-G_2))=0.
	\end{align*}
	Hence, the map
	\begin{align*}
	H^0(X,\OO_X(-3K_X+kG)) \to 
	H^0(G_1, \OO_{G_1}(-3K_{G_1}))\oplus H^0(G_2,  \OO_{G_2}(-3K_{G_2}))
	\end{align*}
	is surjective. 
	Since $|-3K_{G_i}|$ gives a birational map on $G_i$ for 
	$i=1,2$, $|-3K_X+kG|$ gives a birational map on $X$ 
	for all $k\geq 4n+2$.
	
	Suppose now that $\dim Z=2$. 
	Note that $-K_X+k G$ is ample for $k\geq n$, and so is $-K_X|_{G}+k G|_G$.
	For two general fibers $F_1$ and $F_2$ of $f|_G$, and for an integer 
	$k\geq 4n+3$, let us consider the short exact sequence
	\begin{align*}
	0\to {}&\OO_G(-3K_X|_G+kG|_G-F_1-F_2)\to \OO_G(-3K_X|_G+kG|_G)\\
	\to{}& \OO_{F_1}(-3K_{F_1})\oplus   \OO_{F_2}(-3K_{F_2})\to 0.
	\end{align*}
	Since $k\geq 4n+3$,
	$$-3K_X|_G+kG|_G-F_1-F_2-K_G\equiv -4K_X|_G+(k-1)G|_G-F_1-F_2$$
	is ample. Again, by the Kawamata--Viehweg vanishing theorem,
	\begin{align*}
	{}&H^1(G, \OO_G(-3K_X|_G+kG|_G-F_1-F_2))=0.
	\end{align*}
	Hence,
	\begin{align*}
	{}& H^0(G,\OO_G(-3K_X|_G+kG|_G))\\
	\to {}&  H^0(F_1, \OO_{F_1}(-3K_{F_1}))\oplus H^0(F_2,  \OO_{F_2}(-3K_{F_2}))
	\end{align*}
	is surjective. Since $|-3K_{F_i}|$ gives a birational map on 
	$F_i\simeq \bP^1$ for $i=1,2$, $|-3K_X|_G+kG|_G|$ gives a 
	birational map on $G$ for all $k\geq 4n+3$.

	For an  integer $k\geq 4n+3$, consider the short exact sequence
	\begin{align*}
	0\to {}&\OO_X(-3K_X+(k-1)G)\to \OO_X(-3K_X+kG)\\
	\to{}&
	\OO_{G}(-3K_X|_G+kG|_G)\to 0.
	\end{align*}
	Since $-4K_X+(k-1)G$ is ample,  by the Kawamata--Viehweg vanishing theorem,
	\begin{align*}
	{}&H^1(X, \OO_X(-3K_X+(k-1)G))\\
	={}&H^1(X, \OO_X(K_X-4K_X+(k-1)G))=0.
	\end{align*}
	Hence
	\begin{align*}
	{}& H^0(X,\OO_X(-3K_X+kG))\to {} H^0(G,  \OO_{G}(-3K_X|_G+kG|_G))
	\end{align*}
	is surjective.  
	We showed that $|-3K_X|_G+kG|_G|$ gives a 
	birational map on $G$ since $k\geq 4n+3.$
	In particular, $|-3K_X+kG|\neq \emptyset$.
	Hence $|-3K_X+(k+1)G|$ can separate general elements in $|G|$, and
	\begin{align*}
	{}& H^0(X,\OO_X(-3K_X+(k+1)G))\to {} H^0(G,  \OO_{G}(-3K_X|_G+(k+1)G|_G))
	\end{align*}
	is surjective, which
	gives a birational map on $G$. 
	This implies that $|-3K_X+(k+1)G|$ gives a birational map for all $k\geq 4n+3$. 
	The proof is then complete.
\end{proof}

\subsubsection{Boundedness of the volume}

\begin{lem}\label{lem volume}
	In the same setting as Section \ref{setting}, there exists a 
	positive integer $v=v(\epsilon)$ depending only on 
	$\epsilon$ such that $(-K_X+n G)^3\leq v$, where $n$ 
	is the natural number given in Lemma \ref{lemma ample}.
\end{lem}

\begin{proof}
	If $\dim Z=1$,  $G$ is a smooth del Pezzo surface. 
	Note that $\Vol(G, -K_G)=K_G^2\leq 9$.
	By Lemma \ref{lemma volume} and the fact that $-K_X$ is not big,
	\begin{align*}
	0=\Vol(X, -K_X)
	\geq \Vol(X, -K_X+n G)-3n\Vol(G, -K_G).
	\end{align*}
	This implies that $(-K_X+n G)^3\leq 27n.$
	
	Now suppose that $\dim Z=2$.
	As constructed in Section \ref{setting}, the fibration
	$f|_G \colon G\to H$ is a conic bundle from a 
	smooth surface to a smooth curve.
	\begin{claim}\label{claim}
		$\Vol(G, -K_X|_{G}+n G|_G)\leq 8+4(n+1)c$.
	\end{claim}
	\begin{proof}[Proof of Claim \ref{claim}]
		As $-K_X+n G$ is ample, so is $-K_X|_{G}+n G|_G$. 
		Also note that $H^2\leq c$ and $G|_G\equiv (H^2)F$ 
		where $F\simeq \bP^1$ is a general fiber of $f|_G$. Hence
		\begin{align*}
		\Vol(G, -K_X|_{G}+n G|_G){}&=(-K_X|_{G}+n G|_G)^2\\
		{}&=(-K_{G}+(n+1) G|_G)^2\\
		{}&=(K_G)^2-2(n+1)K_G\cdot (H^2)F\\
		{}&\leq 8+4(n+1)c,
		\end{align*}
		where we used the fact that for the conic bundle $G$, 
		$K_G^2\leq 8$.
	\end{proof}
	By Lemma \ref{lemma volume} and Claim \ref{claim},
	\begin{align*}
	0=\Vol(X, -K_X) \geq \Vol(X, -K_X+n G)-3n\Vol(G, -K_X|_{G}+n G|_G).
	\end{align*}
	In particular, $(-K_X+n G)^3\leq 3n(8+4(n+1)c)$.
\end{proof}

\subsection{Proof of Theorems \ref{main} and \ref{main2}}\label{subsec mains}
\begin{proof}[Proof of Theorem \ref{main2}]
	According to Proposition \ref{thm MFS}, after a 
	birational modification, we may assume that 
	$X$ has a terminal Mori fibration structure 
	$f\colon X\to Z$. If $-K_X$ is big, then $X$ 
	is of Fano type, and the (birational) boundedness 
	follows from the solution of (birational) BAB 
	conjecture in dimension $3$, see 
	\cite[Corollary 1.8]{Jiang6} or \cite[Corollary 1.2]{Bir16b}.  
	So we only need to consider the case when 
	$-K_X$ is not big (and $\dim Z>0$).
	By Theorem \ref{NV}, there exist positive integers 
	$n = n(\epsilon)$ and $v = v(\epsilon)$ depending 
	only on $\epsilon$, and an effective Cartier divisor 
	$G$ on $X$, such that $|-3K_X+8nG|$ defines a 
	birational map, and $(-K_X+nG)^3\le v$. 
	Moreover, ${\Vol}(X, -3K_X+8nG)$ is bounded from above by
	\[
	{\Vol}(X, -3K_X+8nG)\le {\Vol}(X, -8K_X+8nG)\le 512v.
	\]
	Therefore, $X$ belongs to a birationally bounded family 
	by the boundedness of Chow varieties (see, for example, 
	\cite[Lemma 2.4.2(2)]{HMX13}).
\end{proof}

\begin{proof}[Proof of Theorem \ref{main}]
	By Theorem \ref{main2}, it suffices to prove that there 
	exists $\epsilon>0$ independent of $X$, such that 
	$X$ is $\epsilon$-klt. This is shown in Lemma 
	\ref{klt=e-klt}, as an easy consequence of the 
	Global ACC and may be well-known to the experts. 
\end{proof}

\begin{lem}\label{klt=e-klt}
	Fix a positive integer $d$. Then there exists a 
	positive real number $\epsilon=\epsilon(d)$ 
	depending only on $d$, such that every 
	$d$-dimensional klt Calabi--Yau variety is $\epsilon$-klt.
\end{lem}

\begin{proof}
	Assume by contradiction that $\{X_i\}$ is a sequence 
	of $d$-dimensional klt Calabi--Yau variety with 
	$\lim_{i\to +\infty}\epsilon_i=0$, where $\epsilon_i>0$ 
	is the minimal log discrepancy of $X_i$. Passing to a 
	subsequence, we may assume that $\epsilon_i$ 
	is decreasing and $\epsilon_i<1$. Let $(X_i',(1-\epsilon_i) D_i')\to X_i$ 
	be the klt pair obtained by extracting a prime divisor of log
	discrepancy $\epsilon_i$. Then 
	$K_{X_i'}+(1-\epsilon_i) D_i'\equiv 0$ and the 
	coefficient of $(1-\epsilon_i) D_i'$ belong to the 
	set $\{1-\epsilon_i\mid i\in\mathbb{N}\}$, which satisfies 
	the descending chain condition and is infinite. 
	This contradicts the Global ACC \cite[Theorem 1.5]{HMX14}.
\end{proof}

\section{Boundedness of Mori fibrations with bounded coefficients}\label{sec 4}
The aim of this section is to prove the following result.

\begin{thm}\label{t.bdd.lcy3fold}
	Fix positive real numbers $\epsilon$, $\delta$, and a positive integer $d$. 
	Assume that Conjecture \ref{conj3} holds in dimension $\leq d-1$.
	Then, the set of log pairs $(X,B)$ satisfying
	\begin{enumerate}
		\item $(X,B)$ is an $\epsilon$-lc log Calabi--Yau 
		pair of dimension $d$,
		\item $X$ is rationally connected,
		\item $B>0$, and the coefficients of $B$ are at least $\delta$,
	\end{enumerate}
	forms a log bounded family modulo flops. 
\end{thm}
As Conjecture \ref{conj3} holds in dimension $\leq 2$, we have the following corollary.
\begin{cor}\label{cor.bdd.lcy3fold}
	Fix positive real numbers $\epsilon$, $\delta$. 
	Then, the the set of log pairs $(X,B)$ satisfying
	\begin{enumerate}
		\item $(X,B)$ is an $\epsilon$-lc log Calabi--Yau 
		pair of dimension $3$,
		\item $X$ is rationally connected,
		\item $B>0$, and the coefficients of $B$ are at least $\delta$,
	\end{enumerate}
	forms a log bounded family modulo flops. 
\end{cor}

Before proceeding to the proof of Theorem \ref{t.bdd.lcy3fold},
we first need some technical results extending those 
contained in \cite{HMX14, DCS}. Such extensions
are made possible by recent work of Birkar solving 
the BAB Conjecture \cite{Bir16a, Bir16b}.

The following lemma can be viewed as a generalization of \cite[Lemma 6.1]{HMX14}
\begin{lem}\label{l.1}
	Fix $0<\epsilon'<\epsilon\leq 1$, and a positive integer $d$.
	Then there is a positive number $t=t(d,\epsilon, \epsilon')$ depending only on $d$, $\epsilon$, and $\epsilon'$, 
	such that if $(X, B)$ is an $\epsilon$-lc pair 
	of dimension $d$ and $(X, \Phi)$ is a log pair such that $\Phi\geq(1-t)B$ and 
	$K_{X}+B\equiv K_{X}+\Phi\equiv 0$, then the pair 
	$(X,\Phi)$ is $\epsilon'$-klt.
\end{lem}

\begin{proof} 
	Assume, to the contrary, that there is a sequence of $d$-dimensional $\epsilon$-lc pairs $(X_i,B_i)$ and log pairs $(X_i, \Phi_i)$ such that 
	$\Phi_i\geq(1-\frac{1}{i})B$ and $K_{X_i}+B_i\equiv K_{X_i}+\Phi_i\equiv 0$, 
	but $(X_i,\Phi_i)$ is not $\epsilon'$-klt, for all integers $i>0$. 
	Replacing $\Phi_i$ by $(1-\lambda_i)B_i+\lambda_i\Phi_i$ for some suitable
	$\lambda_i\in [0,1)$, we can assume that $(X,\Phi_i)$ is $\epsilon'$-lc
	but not $\epsilon'$-klt. 
	
	By \cite[Corollary 1.4.3]{BCHM}, we can take a $\bQ$-factorial birational modification $\phi_i:Y_i\rightarrow X_i$ extracting precisely one irreducible divisor $S_i$ with $a(S_i, X_i,\Phi_i)=\epsilon'$.
	We can write 
	\begin{eqnarray}\label{eq.extraction}
	K_{Y_i}+\phi_{i, \ast}^{-1}\Phi_i + (1-\epsilon')S_i
	=\phi_i^{*}(K_{X_i}+\Phi_i)\equiv 0,
	\end{eqnarray}
	and
	$K_{Y_i}+\phi_{i, \ast}^{-1}B_i+a_i S_i=\phi_i^*({X_i}+B_i)\equiv 0$, 
	where $a_i\leq 1-\epsilon$.
	Now $K_{Y_i}+\phi_{i, \ast}^{-1}\Phi_i $ is not pseudo-effective, so we can 
	run a $(K_{Y_i}+\phi_{i, \ast}^{-1}\Phi_i)$-MMP with scaling of an ample divisor, which ends with a Mori fiber space $\pi_i: W_i \to Z_i$. As this MMP is also a $(-S_i)$-MMP by \eqref{eq.extraction}, $S_i$ dominates $Z_i$.
	
	Denote by $\Phi_i'$, $B_i'$, and $S_i'$ the strict transform of $\phi_{i, \ast}^{-1}\Phi_i$, $\phi_{i, \ast}^{-1}B_i$, and $S_i$ on $W_i$ respectively. Then by construction, $(W_i, \Phi_i' + (1-\epsilon')S_i')$ is an $\epsilon'$-lc log Calabi--Yau pair, $K_{W_i}+B_i' + a_iS_i'\equiv 0$ with $a_i\leq 1-\epsilon$, and $\Phi_i'\geq (1-\frac{1}{i})B_i'$.
	
	Denote by $F_i$ a general fiber of $\pi_i$, then $F_i$ is an  $\epsilon'$-lc Fano variety of dimension $\leq d$. The family of such $F_i$ is bounded by the BAB theorem \cite{Bir16b}, hence there exist positive integers $r$ and $M$ such that $rK_{F_i}$ is Cartier and $(-K_{F_i})^{d_i}\leq M$ where $d_i=\dim F_i$.
	Now consider 
	\begin{align*}
	0\leq {}&\Phi_i'|_{F_i}- \left(1-\frac{1}{i}\right)B_i'|_{F_i}\\
	\equiv {}&-K_{F_i}-(1-\epsilon')S_i'|_{F_i}+\left(1-\frac{1}{i}\right)K_{F_i}+\left(1-\frac{1}{i}\right)a_i S_i'|_{F_i}\\
	\equiv{}&-\frac{1}{i}K_{F_i}-\left(1-\epsilon'-\left(1-\frac{1}{i}\right)a_i\right)S_i'|_{F_i},
	\end{align*}
	we have
	\begin{align*}
	\frac{M}{i}\geq  {}&\frac{1}{i}(-K_{F_i})^{d_i}\\
	\geq  {}&\left(1-\epsilon'-\left(1-\frac{1}{i}\right)a_i\right)S_i'|_{F_i}\cdot (-K_{F_i})^{d_i-1}\\
	\geq {}&(\epsilon-\epsilon')S_i'|_{F_i}\cdot (-K_{F_i})^{d_i-1}\\
	\geq {}&\frac{\epsilon-\epsilon'}{r^{d_i-1}}\geq \frac{\epsilon-\epsilon'}{r^{d-1}}.
	\end{align*}
	Note that as $\frac{\epsilon-\epsilon'}{r^{d-1}}$ is a constant positive number, this is absurd.
\end{proof}

\begin{lem}\label{l.2}
	Fix a positive real number $\epsilon$ and a positive integer $d$. Then there exists a number $M=M(\epsilon, d)$ depending only on $\epsilon$ and $d$, such that if $(X,B)$ is an $\epsilon$-lc log Calabi--Yau pair of dimension $d$, then ${\rm Vol}(X,B)\leq M$.
\end{lem}

\begin{proof}
	We may assume that ${\rm Vol}(X,B)>0$; otherwise it is clear. Then $B$ is big and 
	this follows from \cite[Corollary 1.2]{Bir16b}.
\end{proof}

\begin{thm}\label{t.ind.step.bdd.mfs}
	Fix a positive real number $\epsilon$ and a positive integer $d$. Then there exists a positive number $k=k(\epsilon,d)$ depending only on  $\epsilon$ and $d$ satisfying the following: if $(X, B)$ is a log pair such that
	\begin{enumerate}
		\item $(X,B)$ is $\epsilon$-lc of dimension $d$,
		\item there exists a contraction of normal varieties $f$: $X\rightarrow Y$ such that $0<\dim Y<d$,
		\item $K_X+B\sim_{\bR} f^*H$ for some very ample divisor $H$ on $Y$, and
		\item $K_Y+H$ is big,
	\end{enumerate}
	then ${\rm Vol}(X,B)\leq k {\rm Vol}(Y, H)$.
\end{thm}

\begin{proof}
	Let $t=t(d,\epsilon,\frac{\epsilon}{2})$ as in Lemma \ref{l.1}. Let $F$ be a general fiber of $f$. We may assume that $H$ is general in its linear system.
	
	First, we claim that ${\rm Vol}(X,tB-2f^*H)=0$. Assume not, then $tB-2f^*H$ is big.
	So there exists an effective $\bR$-divisor $E\sim_{\bR}tB-2f^*H$. Let $\Phi=(1-t)B+E$. Then $\Phi|_F\geq(1-t)B|_F$ and $K_X+\Phi\sim_{\bR} -f^*H$. By Lemma \ref{l.1}, $(X,\Phi)$ is klt over the generic point of $Y$. So by Ambro's canonical bundle formula (see \cite[Lemma 2.14]{DCS}), there are pseudo-effective divisors $B_Y$ and $M_Y$ such that $- f^*H \sim_{\bR} K_X+\Phi \sim_{\bR} f^*(K_Y+B_Y+M_Y)$. This immediately gives a contradiction since $K_Y+H$ is big.
	
	Now by Lemma \ref{lemma volume}, we have
	\[
	0={\rm Vol}(X,tB-2f^*H)\geq {\rm Vol}(X,tB)-2d{\rm Vol}(f^*H,tB|_{f^*H}).
	\]
	Hence, ${\rm Vol}(X,B)\leq 2d t^{-1}{\rm Vol}(f^*H,B|_{f^*H})$.
	
	If $\dim Y=1$, then $f^*H=\sum_{i=1}^hF_i$ where $h=\deg(H)={\rm Vol}(Y, H)$ and for each $i$, $F_i$ is a general fiber of $f$ and hence $(F_i, B|_{F_i})$ is an $\epsilon$-lc log Calabi--Yau pair of dimension $d-1$. Hence by Lemma \ref{l.2}, we have 
	\[
	{\rm Vol}(f^*H,B|_{f^*H})=\sum_{i=1}^h{\rm Vol}(F_i,B|_{F_i})\leq M(\epsilon,d-1){\rm Vol}(Y, H).
	\] 
	Hence we may take $k(\epsilon,d)=2dt^{-1} M(\epsilon,d-1)$   in this case. In particular, this settles the case that $\dim X=2$.
	
	Finally we will use induction on the dimension of $X$ to show the case $\dim Y>1$.
	If $\dim Y>1$, consider the map $f_1=f|_{f^*H}$: $f^*H\rightarrow H$. Then $({f^*H}, B|_{f^*H})$ is $\epsilon$-lc of dimension $d-1$,
	$K_{f^*H}+B|_{f^*H}\sim_\bR 2f^*H|_{f^*H}=2(f_1)^*(H|_{H})$, and $K_H+2H|_H$ is big. Hence by the inductive hypothesis, we have
	\[
	{\rm Vol}(f^*H,B|_{f^*H})\leq k(\epsilon,d-1){\rm Vol}(H,2H|_H)=2^{\dim Y-1}k(\epsilon,d-1){\rm Vol}(Y,H).
	\]
	Taking $k(\epsilon,d):=2^{d}d t^{-1} k(\epsilon,d-1)$, the proof is complete.
\end{proof}

The following theorem can be viewed as a relative version of the Special BAB \cite[Theorem 1.4]{Bir16a}.
\begin{thm}\label{t.bdd.mfs}
	Fix a positive integer $d$ and positive numbers $\epsilon$, $\delta$, and $M$. Then the set of log pairs $(X,B)$ satisfying
	\begin{enumerate}
		\item $(X,B)$ is an  $\epsilon$-lc log Calabi--Yau pair of dimension $d$ such that $K_X$ is $\bQ$-Cartier,
		\item there is a contraction of normal varieties $f$: $X\rightarrow Y$ such that $-K_X$ is ample over $Y$ and $0<\dim Y<d$,
		\item there is a very ample Cartier divisor $H$ with ${\rm Vol}(Y,H)\leq M$, and
		\item the coefficients of $B$ are at least $\delta$
	\end{enumerate}
	forms a log bounded family.
\end{thm}

\begin{proof}We may assume that $\delta<1$. Fix a positive integer $n>\frac{1}{\epsilon}$.
	After replacing $H$ by a fixed multiple depending only on $\dim Y$, we may assume that $K_Y+\delta H$ is big, cf. \cite[Lemma 2.3.4(2)]{HMX13}. 
	We will always consider $H$ as a general member of its linear system. 
	As $f^*H$ is base-point free, we can find an effective $\bQ$-divisor $H'\sim_{\mathbb{Q}}f^*H$ with all coefficients equal to $\frac{1}{n}$ such that $(X,B+H')$ is still $\epsilon$-lc. Let $\pi: \tilde{X}\to X$ 
	be a log resolution of $(X,B+H')$ and write 
	\[
	K_{\tilde{X}}+\tilde{B}+\tilde{H}+E=\pi^*(K_X+B+H')+\sum a_iE_i,
	\]
	where $\tilde{B}$ and $\tilde{H}$ are the strict transform of $B$ and $H$ respectively,
	$E_i$ are prime $\pi$-exceptional divisors and $E=\sum E_i$. 
	We can choose a sufficiently small positive number  $s$ (depending on $X$) such that $(X,(1+s)B+H')$ is $\frac{\epsilon}{2}$-lc and $K_X+(1+s)B+H'\equiv -sK_X+H'$ is ample. Note that here we need the assumption that $K_X$ is $\bQ$-Cartier. 
	So
	\begin{align*}
	0&<{\rm Vol}(K_X+(1+s)B+H')\leq {\rm Vol}(K_{\tilde{X}}+\roundup{\tilde{B}}+\tilde{H}+E) \\
	&\leq  {\rm Vol}(K_{X}+\roundup{{B}}+H')\leq {\rm Vol}\left(\frac{1}{\delta}B+H'\right)\\
	&= \frac{1}{\delta^d}{\rm Vol}\left(B+\delta H'\right)<\frac{1}{\delta^d}{\rm Vol}\left(B+ H'\right)
	\end{align*}
	Note that $K_X+B+ H'\sim_\bR f^*( H)$, hence $\frac{1}{\delta^d}{\rm Vol}\left(B+ H'\right)$ is bounded from above by Theorem \ref{t.ind.step.bdd.mfs}. 
	
	In summary, $(\tilde{X}, \roundup{\tilde{B}}+\tilde{H}+E)$ is an lc pair such that $K_{\tilde{X}}+\roundup{\tilde{B}}+\tilde{H}+E$ is big with bounded volume and the coefficients of $\roundup{\tilde{B}}+\tilde{H}+E$ are in the fixed finite set $\{\frac{1}{n}, 1\}$ independent of $X$. The set of such pairs forms a log birationally bounded set by \cite[Lemma 2.3.4(2), Theorem 3.1]{HMX13} and \cite[Theorem 1.3]{HMX14}. 
	Therefore, the set $\{(X,(1+s)B+H')\}$ is also log birationally bounded, and hence log bounded by \cite[Theorem 1.6]{HMX14}. So the set $\{(X,B)\}$ is log bounded.
\end{proof} 
It would be interesting to ask whether Theorem \ref{t.bdd.mfs} still holds true if we relax the condition ``$-K_X$ is ample over $Y$'' to ``$-K_X$ is big over $Y$''. Here, as a corollary, we can prove a weak version, namely, $(X,B)$ forms a log bounded family  modulo flops.
\begin{cor}\label{cor.bdd.mfs}
	Fix a positive integer $d$ and positive numbers $\epsilon$, $\delta$, and $M$. 
	Then the set of log pairs $(X,B)$ satisfying
	\begin{enumerate}
		\item $(X,B)$ is an  $\epsilon$-lc log Calabi--Yau pair of dimension $d$,
		\item there is a contraction of normal varieties $f\colon X\rightarrow Y$ 
		such that $-K_X$ is big over $Y$ and $0<\dim Y<d$,
		\item there is a very ample Cartier divisor $H$ with ${\rm Vol}(Y,H)\leq M$, and
		\item the coefficients of $B$ are at least $\delta$
	\end{enumerate}
	forms a log bounded family modulo flops.
\end{cor}

\begin{proof}We may replace $X$ by its $\bQ$-factorialization and assume that $X$ is $\bQ$-factorial.
	Since $(X, B)$ is an $\epsilon$-lc pair, 
	there exists a sufficiently small $t>0$ such that 
	$(X, (1+t)B)$ is also klt. Since $B$ is big over $Y$, 
	we may run a $(K_X+(1+t)B)$-MMP over $Y$ 
	(which is also a $B$-MMP over $Y$) with scaling 
	of an ample divisor, and finally reach a relative 
	log canonical model $f' \colon X' \to Y$. Denote by 
	$B'$ the strict transform of $B$.
	It follows that $-K_{X'}\equiv B'$ is ample over $Y$. 
	Also note that $(X', B')$ is again an 
	$\epsilon$-lc log Calabi--Yau pair  
	and the coefficients of $B'$ are at least $\delta$. 
	Then, by Theorem \ref{t.bdd.mfs}, $(X', B')$ belongs to a 
	log bounded family. The conclusion then follows from 
	Proposition \ref{p.blow.ups.fam}, as 
	for any prime divisor $E$ on $X$ which is exceptional over $X'$,
	we have
	\[
	a(E, X', B')=a(E, X, B)\leq a(E, X, 0)=1,
	\]
	where the first equality follows from the fact that $(X, B)$
	and $(X', B')$ are crepant birational log Calabi--Yau pairs.
\end{proof}

\begin{proof}[Proof of Theorem \ref{t.bdd.lcy3fold}]
	We follow the strategy of \cite{DCS}. We may replace $X$ by its $\bQ$-factorialization and assume that $X$ is $\bQ$-factorial.
	Since $K_X+B \equiv 0$ and $B>0$, we can run a $K_X$-MMP 
	with scaling of an ample divisor
	which ends with a Mori fiber space $f:Y \to Z$ with general fiber $F$.
	Denote by $B_{Y}$ the strict transform of $B$.
	Since $B>0, \; K_X+B \equiv 0$, and we are running a $K_X$-MMP, 
	it follows that $B_{Y}>0$. Also note that $(Y, B_Y)$ is again 
	an $\epsilon$-lc log Calabi--Yau pair, $Y$ is rationally connected, 
	and the coefficients of $B_Y$ are at least $\delta$. 
	
	If $\dim Z =0$, then $Y$ is an $\epsilon$-lc log Fano variety, hence it is 
	bounded by \cite[Theorem 1.1]{Bir16b}. It is then easy to show that the support
	of $B_{Y}$ is also bounded, since there exist positive integers $r$ and 
	$M$ such that $-rK_Y$ is very ample and $(-K_Y)^d\leq M$, and therefore
	\[
	\Supp(B_Y)\cdot (-rK_Y)^{d-1}\leq \frac{1}{\delta}B_Y\cdot(-rK_Y)^{d-1}
	\leq \frac{r^{d-1}M}{\delta}.
	\]
	Hence $(Y, B_Y)$ is log bounded.
	
	If $\dim Z >0$, by Ambro's canonical bundle formula  \cite[Theorem 3.1]{FG}, $Z$
	is naturally endowed with a log Calabi--Yau structure, that is, there exists
	an effective $\bR$-divisor $\Gamma$ on $Z$ such that $(Z, \Gamma)$ is klt 
	and $K_Z+\Gamma \equiv 0$.
	Denote by $F$ a general fiber of $f$. Then $F$ is an $\epsilon$-lc log Fano variety of 
	dimension is at most $d$, $K_F+B_Y|_F\equiv0$ and the coefficients of $B_Y|_F$ are at 
	least $\delta$. Again by \cite[Theorem 1.1]{Bir16b}, $(F, B_Y|_F)$ is log bounded. Hence by 
	\cite[Theorem 1.4]{Birkar}, $\Gamma$ can be chosen so that 
	$(Z, \Gamma)$ is $\epsilon'$-lc for some $\epsilon'$ which only depends on 
	$\epsilon$ and $d$. In particular, by Conjecture \ref{conj3}, it follows that $Z$
	belongs to a bounded family since $Z$ is rationally connected.
	By Theorem \ref{t.bdd.mfs}, $(Y, B_Y)$ is log bounded.
	
	In summary, $(Y, B_Y)$ belongs to a log bounded family.
	For any prime divisor $E$ on $X$ which is exceptional over $Y$,
	we have
	\[
	a(E, Y, B_Y)=a(E, X, B)\leq a(E, X, 0)=1.
	\]
	The conclusion then follows from Proposition \ref{p.blow.ups.fam}.
\end{proof}

The following proposition is a generalization of \cite[Proposition 2.5]{HX15}.

\begin{prop}\label{p.blow.ups.fam}
	Fix a positive real number $\epsilon$ and a positive integer $d$.
	Let $\mathcal{D}$ be a log bounded family of log pairs such 
	that any $(X, B) \in \mathcal{D}$ is a $d$-dimensional $\epsilon$-lc pair.
	\\
	Then there exist finitely many quasi-projective normal $\bQ$-factorial varieties
	$\mathcal{Y}_i$, a reduced divisor $\mathcal{F}_i$ on $\mathcal Y_i$, 
	and a projective morphism $\mathcal{Y}_i\to T_i$, where $T_i$ is a 
	normal variety and $\mathcal{F}_i$ does not contain 
	any fiber, such that for any $(X, B)\in \mathcal{D}$, and any set of 
	divisors $\{E_{j}\}$ exceptional over $X$ such that the log discrepancy 
	$a(E_{j},X,B)\leq 1$, there exists an index $i$, a closed point 
	$t\in T_i$, and a birational morphism $\mu _t : \mathcal Y_{i, t} \to X$ 
	which extracts precisely the divisors $\{E_{j}\}$, and  $\mathcal{F}_{i,t}$ 
	coincides with the support of strict transform of $B$ and every $E_i$.
\end{prop}

\begin{proof}
	By definition, there is a  quasi-projective scheme $\mathcal{Z}$, a reduced divisor $\mathcal{E}$ on $\mathcal Z$, and a projective morphism $h: \mathcal{Z}\to T$, where $T$ is of finite type and $\mathcal{E}$ does not contain any fiber, such that
	for every $(X, B)\in \mathcal{D}$, there is a closed point $t \in T$ and an isomorphism $f : \mathcal{Z}_t \to X$
	such that $\mathcal{E}_t$  coincides with the support of $f_*^{-1}B$.
	
	We may assume that $T$ is reduced.  Blowing up $(\mathcal{Z}, \mathcal{E})$ and $T$ and decomposing $T$ into a finite union of
	locally closed subsets, we may assume that there exists a pair $(\mathcal{Z}', \mathcal{E}')$ that has simple normal crossings support with the following diagram:
	\begin{displaymath}
	\xymatrix{
		(\mathcal{Z}', \mathcal{E}') \ar[rr] \ar[dr] & & (\mathcal{Z}, \mathcal{E}) \ar[dl]\\
		&T &
	}
	\end{displaymath}
	Passing to
	an open dense subset of $T$, we may assume that the fibers of $(\mathcal{Z}',  \mathcal{E}') \to T$ are log smooth pairs, passing to a finite cover of $T$, we may
	assume that every stratum of $(\mathcal{Z}',  \mathcal{E}')$ has irreducible fibers over $T$; decomposing $T$ into
	a finite union of locally closed subsets, we may assume that $T$ is smooth; finally
	passing to a connected component of $T$, we may assume that $T$ is integral.
	Note that all these operations will still yield a bounded family, thanks to Noetherian induction.
	
	As $( \mathcal{Z}',(1-\epsilon)\mathcal{E}')$ is klt, it follows that there are only finitely many valuations of log
	discrepancy at most one with respect to $( \mathcal{Z}',(1-\epsilon)\mathcal{E}')$.  
	As $( \mathcal{Z}',(1-\epsilon)\mathcal{E}')$ has simple normal crossings over $T$, there is a sequence of blow ups $\phi: \mathcal{Z}'' \to \mathcal{Z}'$ of strata of $\mathcal{E}'$ extracting all and only the exceptional valuations of log discrepancy at most one for $( \mathcal{Z}',(1-\epsilon)\mathcal{E}')$.  
	Note also that as $( \mathcal{Z}',(1-\epsilon)\mathcal{E}')$ has simple normal
	crossings over $T$, it follows that if $t\in T$ is a closed point then the center on $\mathcal{Z}''_t$ of any valuation of log discrepancy at most one with respect to $( \mathcal{Z}'_t,(1-\epsilon)\mathcal{E}'_t)$ has codimension 1.
	
	Let us denote by $\mathcal{E}''$ the support of the strict transform of $\mathcal{E}'$ and all $\phi$-exceptional divisors on $\mathcal{Z}''$. 
	Let us note that $K_{\mathcal{Z}''}\sim_{\bQ} \mathcal{A}+ \mathcal{C}$, where $ \mathcal{A}$ is a general $\bQ$-Cartier divisor ample  over $\mathcal{Z}$, and $\mathcal{C}$ is an effective $\bQ$-divisor. 
	Moreover, we can find a positive real number $\lambda$ such that 
	$
	(\mathcal{Z}'', (1+\lambda)(1-\frac{\epsilon}{2})\mathcal{E}''+\lambda(\mathcal{A}+ \mathcal{C}))
	$
	is still klt.
	
	Let $\mathcal{E}''_k$ for $1\leq k\leq m$ be the components of $\mathcal{E}''$. We may assume that $\mathcal{E}''_k$ for $1\leq k\leq r$ are the components which are the strict transform of components of $\mathcal{E}$ on $\mathcal{Z}''$. 
	Consider the polytope $$P=\left\{\sum_{k=1}^m h_k  \mathcal{E}''_k \; \mid
	\;  0\leq h_k \leq 1-\frac{\epsilon}{2} \; \text{for }1\leq  k\leq m\right\}.$$
	Note that for any element $\Phi$ of $P$, 
	a $(K_{ \mathcal{Z}''} + \Phi)$-minimal model over $\mathcal{Z}$ is the same as a  $(K_{ \mathcal{Z}''} + (1+\lambda)\Phi+\lambda(\mathcal{A}+ \mathcal{C}))$-minimal model over $\mathcal{Z}$. Hence by the finiteness of models \cite[Corollary 1.1.5]{BCHM}, 
	there are just finitely many possible minimal models over $\mathcal{Z}$ for this polytope $P$, that is, there exist finitely many quasi-projective normal $\bQ$-factorial varieties $\mathcal{Y}_i\to T$,  such that for any element $\Phi$ of $P$, there exists an index $i$ and a closed point $t\in T$ such that $\mathcal{Y}_{i,t}$ is a $(K_{ \mathcal{Z}''} + \Phi)$-minimal model over $\mathcal{Z}$.
	Denote $\mathcal{F}_i$ to be the strict transform of  $\mathcal{E}''$ on $\mathcal{Y}_i$.
	
	
	
	We now claim that such $(\mathcal{Y}_i, \mathcal{F}_i)$ are what we need.
	For $(X,B)\in \mathfrak{D}$, there is a closed point $t\in T$ and a
	birational map $f''\colon  \mathcal{Z}''_t \to X$ such that the support of $\mathcal{E}''_t$ contains the support
	of the strict transform of $B$ and all $f''^{-1}$-exceptional divisors.  
	As $(X, B)$ is $\epsilon$-lc, we have 
	\[
	K_{ \mathcal{Z}'_t}+\Delta'_t=f'^\ast(K_X+B),
	\]
	\[
	K_{ \mathcal{Z}''_t}+\Delta''_t=f''^\ast(K_X+B)
	\]
	where $f':\mathcal{Z}'_t\to X$, $\Delta'_t\leq (1-\epsilon)\mathcal{E}'_t$ and $\Delta''_t\leq (1-\epsilon)\mathcal{E}''_t$.
	Now, we consider the set $\{E_j\}$ of all divisors which are exceptional over $X$ and such that $a(E_j, X, B)\leq 1$, then 
	\[
	1\geq a(E_j, X, B)=a(E_j,\mathcal{Z}'_t, \Delta'_t)\geq a(E_j,\mathcal{Z}'_t, (1-\epsilon)\mathcal{E}'_t).
	\]
	By the construction of $\mathcal{Z}''$, each $E_j$ is realized as a divisor center on $\mathcal{Z}''_t$ and it appears as a component of $\mathcal{E}''_t$. 
	We may assume that $\mathcal{E}''_{k, t}$ for $r<k\leq n$ correspond to the divisors $E_j$ on $\mathcal{Z}''_t$. 
	Now, we may write 
	$$
	\Delta''_t=\sum_{k=1}^m b_k\mathcal{E}''_{k}
	$$
	where $b_k\in [0, 1-\epsilon]$ for $1\leq k\leq n$ and $b_k\leq 1-\epsilon$ for $k>n$.
	We can consider the boundary
	$$
	\Phi=\sum_{k=1}^n b_k\mathcal{E}''_{k}+\sum_{n<k\leq m}\left(1-\frac{\epsilon}{2}\right)\mathcal{E}''_{k}\in P
	$$
	on $\mathcal{Z}''$. Then by construction, there exists an index $i$ such that $\mathcal{Y}_i$ is a $(K_{\mathcal{Z}''}+\Phi)$-minimal model over $\mathcal{Z}$, and hence $\mathcal{Y}_{i,t}$ is a $(K_{\mathcal{Z}''_t}+\Phi_t)$-minimal model over $\mathcal{Z}_t\cong X$, which extracts precisely the divisors  $\{\mathcal{E}''_{k,t}\mid r<k\leq n\}=\{E_j\}$.
\end{proof}

\section{Towards boundedness of rationally connected klt Calabi--Yau's}\label{section 7}

We have seen so far that rationally connected klt Calabi--Yau $3$-folds
are birationally bounded in Section \ref{sec 3}. By applying results in Section \ref{sec 4}, in a good number of cases, it is actually possible to prove that boundedness holds
under slightly stronger assumptions.

Firstly we show that rationally connected klt Calabi--Yau $3$-folds with mld bounded away from $1$ are  bounded modulo flops.
\begin{thm}\label{t.bdd.mld}
	Fix $0<c<1$.
	Let $\mathcal{D}$ be the set of varieties $X$ such that
	\begin{enumerate}
		\item $X$ is a rationally connected klt Calabi--Yau $3$-fold, and
		\item $0<\mld(X)< c$.
	\end{enumerate}
	Then $\mathcal{D}$ is   bounded modulo flops. 
\end{thm}

\begin{proof}
	Take $X\in \mathcal{D}$. 
	By \cite[Corollary 1.4.3]{BCHM}, we may take   a birational morphism $\pi \colon Y \to X$ extracting only one exceptional 
	divisor $E$ of log discrepancy $a=a(E, X) \in (0, c)$. Then
	\[
	K_Y+(1-a)E = \pi^\ast K_X.
	\]
	Also by Global ACC (see Lemma \ref{klt=e-klt}), there exists a constant $\epsilon \in (0,\frac{1}{2})$ such that $X$ is $(2\epsilon)$-lc, and therefore $(Y, (1-a)E)$ is a $(2\epsilon)$-lc log Calabi--Yau pair with $1-a>1-c>0$ and $Y$ rationally connected. 
	
	Now  by Corollary \ref{cor.bdd.lcy3fold}, the pairs $(Y, E)$ are log bounded modulo flops.
	That is, there are finitely many quasi-projective normal varieties $\mathcal{W}_i$, a reduced divisor $\mathcal{E}_i$ on $\mathcal W_i$, and a projective morphism $\mathcal{W}_i\to S_i$, where $S_i$ is a normal variety of finite type and $\mathcal{E}_i$ does not contain any fiber, such that
	for every $(Y, E)$, there is an index $i$, a closed point $s \in S_i$, and a small birational
	map $f : {\mathcal{W}_{i,s}} \dashrightarrow Y$
	such that $\mathcal{E}_{i,s}=f_*^{-1}E$. We may assume that the set of points $s$ corresponding to such $Y$ is dense in each $S_i$.
	We may just consider a fixed index $i$ and ignore the index in the following argument.
	
	For the point $s$ corresponding to $(Y,E)$, $$K_{{\mathcal{W}_s}}+(1-a)f_*^{-1}E\equiv f_*^{-1}(K_Y+(1-a)E)\equiv 0$$ and therefore $({\mathcal{W}_s}, (1-a)f_*^{-1}E)$ is a $(2\epsilon)$-lc log Calabi--Yau pair. 
	Now consider a log resolution $g: \mathcal{W}'\to \mathcal{W}$ of $(\mathcal{W},\mathcal{E})$ and denote by $\mathcal{E}'$ the strict transform of $\mathcal{E}$ and all the sum of $g$-exceptional reduced divisors on $\mathcal{W}'$. Consider the log canonical pair $(\mathcal{W}',(1-\epsilon)\mathcal{E}')$. 
	There exists an open dense set $U\subset S$ such that for the point $s\in U$ corresponding to $(Y,E)$, 
	$g_s: \mathcal{W}'_s\to \mathcal{W}_s$ is a log resolution and we can write
	$$
	K_{\mathcal{W}'_s}+B_s=g_s^*(K_{{\mathcal{W}_s}}+(1-a)f_*^{-1}E)\equiv 0
	$$
	where the coefficients of $B_s$ are $\leq 1-2\epsilon$ and its support is contained in $\mathcal{E}'_s=\mathcal{E}'|_{{\mathcal{W}'_s}}$. We have
	\begin{align*}
	{}&(K_{\mathcal{W}'}+(1-\epsilon)\mathcal{E}')|_{{\mathcal{W}'_s}}\\
	\equiv {}&K_{\mathcal{W}'_s}+(1-\epsilon)\mathcal{E}'_s\\
	\equiv {}&(1-\epsilon)\mathcal{E}'_s-B_s.
	\end{align*}
	Note that the support of $(1-\epsilon)\mathcal{E}'_s-B_s$ coincides with the support of $\mathcal{E}'_s$ which are precisely the divisors on $\mathcal{W}'_s$ exceptional over $X$. Hence $(K_{\mathcal{W}'}+(1-\epsilon)\mathcal{E}')$ is of Kodaira dimension zero on the fiber $\mathcal{W}'_s$ and we can run a $(K_{\mathcal{W}'}+(1-\epsilon)\mathcal{E}')$-MMP with scaling of an ample divisor over $S$ to get a relative minimal model $\tilde{\mathcal{W}}$ over $S$. This MMP terminates by \cite[Corollary 2.9, Theorem 2.12]{HX13}. Note that for the point $s\in U$ corresponding to $(Y,E)$, $\mathcal{E}'_s$ is contracted and hence $\tilde{\mathcal{W}}_s$ is isomorphic to $X$ in codimension one. This gives a bounded family modulo flops, over $U$.
	Applying Noetherian induction on $S$, the family of all such $X$ is bounded modulo flops.
\end{proof}

As an interesting application, we can show that for rationally connected klt Calabi--Yau $3$-folds, boundedness modulo flops is equivalent to the boundedness of the global index.

\begin{cor}\label{c.bdd.index}
	Let $\mathcal{D}$ be a set of rationally connected klt Calabi--Yau $3$-folds.
	Then $\mathcal{D}$ is  bounded modulo flops if and only if
	there exists a positive integer $r$ such that $rK_X\sim 0$ for any $X\in \mathcal{D}$.
\end{cor}

\begin{proof}
	Assume that $\mathcal{D}$ is bounded modulo flops, then 
	there exists a bounded family $\mathcal{D}'$ of normal projective varieties such that for every $X\in \mathcal{D}$, there exists $Y\in\mathcal{D}'$ and a small birational morphism $f: Y\dashrightarrow X$. Moreover, $K_Y$ is $\bQ$-Cartier by definition. Hence we have $K_Y=f_*^{-1}K_X\sim_\bQ 0$.
	Take a common resolution $p:W\to Y$ and $q:W\to X$,
	by the Negativity Lemma, we have $p^*K_Y=q^*K_X$ and $Y$ is klt.
	Since $\mathcal{D}'$ is bounded, there exists a constant $r \in \mathbb{N}_{>0}$ such that $rK_Y$ is Cartier, which means that $p^*(rK_Y)\sim_\bQ 0$ is a Cartier divisor. Since $W$ is rationally connected, it is simply connected and hence $p^*(rK_Y)\sim 0$. Therefore $rK_X=q_*q^*(rK_X)\sim 0$. This proves the `only if' part.
	
	
	If there exists a positive integer $r$ such that $rK_X\sim 0$ for any $X\in \mathcal{D}$, then it is clear that $r\cdot \mld(X)$ is a positive integer.  Since $X$ is rationally connected, $X$ has worse than canonical singularities, that is, $\mld(X)<1$. Therefore $\mld(X)\leq 1-\frac{1}{r}<1-\frac{1}{2r}$. 
	The `if' part follows directly from Theorem \ref{t.bdd.mld}.
\end{proof}

As another application, Theorem \ref{t.bdd.mld} relates Conjecture \ref{conj1} to conjectures for minimal log discrepancies. Recall the following deep conjecture regarding the behavior of minimal log discrepancy proposed by Shokurov. 
\begin{conj}[ACC for mld, cf. {\cite[Problem 5]{Sho88}, \cite[Conjecture 4.2]{MR1420223}}]\label{c.mld.strong}
	Fix a positive integer $d$ and a DCC set $I \subset [0,1]$. Then the set
	\[
	\{\mathrm{mld}_{\eta_Z}(X, \Delta) \mid  (X, \Delta) \text{ is lc}, \; \dim X\leq d, \;
	Z\subset X, \; \mathrm{coeff}(\Delta) \in I \}
	\]
	satisfies the ACC. 
\end{conj}

ACC stands for ascending chain condition whilst DCC stands for descending chain condition. For recent progress on minimal log discrepancies, we refer the readers to \cite{MN18,Kaw18,Liu18,HLS19,HL20,NS20}.

Here we only need a very weak version of Conjecture \ref{c.mld.strong}.

\begin{conj}[1-Gap conjecture for mld]\label{c.mld}Fix a positive integer $d$.
	Then $1$ is not an accumulation point from below for the set
	\[
	\{\mathrm{mld}(X)  \mid \dim X\leq d\}.
	\]
\end{conj}

\begin{cor}\label{t.bdd.cyklt}
	Assume that Conjecture \ref{c.mld} holds for rationally connected klt Calabi--Yau $3$-folds. Then the set of  rationally connected klt Calabi--Yau $3$-folds is bounded modulo flops. %
\end{cor}

\begin{proof}
	Let $X$ be a rationally connected klt Calabi--Yau $3$-fold. Note that $X$ has worse than canonical singularities, that is, $\mld(X)<1$. 
	By Conjecture \ref{c.mld}, there exists a constant $c \in (0,1)$ such that $\mld(X)< c$.
	Hence the corollary follows directly from Theorem \ref{t.bdd.mld}. 
\end{proof}


Finally we consider  rationally connected klt Calabi--Yau $3$-folds with positive augmented irregularity. For klt Calabi--Yau varieties, Greb, Guenancia, and Kebekus 
proved that there is a decomposition,
after a quasi-\'etale covering, that is analogous to the classical
Beauville--Bogomolov decomposition in the smooth case, \cite{MR730926}.

\begin{thm}[{\cite[Theorem B]{ggk}}]\label{t.ggk}
	Let $X$ be a klt variety with $K_X\equiv 0$. Let $H$ be an ample divisor on $X$, and $\omega_H\in c_1(H)$ the singular Ricci-flat K\"ahler metric constructed by \cite{EGZ09}. Then there are normal projective varieties $A, Z$ and a quasi-\'etale cover $\gamma \colon A\times Z \to X$ such that 
	\begin{enumerate}
		\item $A$ is an abelian variety of dimension $\tilde{q}(X)$,
		\item $Z$ has canonical singularities, $K_Z\sim 0$ and 
		$\tilde{q}(Z)=0$, and
		\item there exists a flat K\"ahler form $\omega_A$ on $A$
		and a singular Ricci-flat K\"ahler metric $\omega_Z$ on $Z$
		such that $\gamma^\ast \omega_H = \mathrm{pr}_1^\ast 
		\omega_A + \mathrm{pr}_2^\ast \omega_Z$ and such that
		the holonomy group of the corresponding Riemannian metric 
		on $A \times Z_{reg}$ is connected.
	\end{enumerate}
\end{thm}

Here we recall the notion of augmented irregularity. 

\begin{definition}[{cf. \cite[Definition 2.20]{ggk}}]\label{d.gen.irr}
	Let $X$ be a normal projective variety. The {\it irregularity} of $X$ is defined
	as $q(X):= h^1(X, \mathcal{O}_X)$.
	The {\it augmented irregularity} of $X$ is defined as $\tilde{q}(X)= \sup \{ q(Y)\; | \;  Y \to X \; \text{quasi-\'etale cover}\} \in \mathbb{N} \cup \{\infty\}$.
\end{definition}

\begin{remark}
	When $X$ has rational singularities (as will be the case
	in this section), then $q(X) = q(Z)$ for any $Z \to X$ resolution
	of singularities.
\end{remark}

As a consequence of the decomposition for singular Calabi--Yau 
varieties and Theorem \ref{t.bdd.mld}, we can prove boundedness for
those rationally connected klt Calabi--Yau $3$-folds
that contain an abelian factor.

\begin{cor}\label{t.bdd.ab.cy}
	Let $\mathcal{D}_{ab}$ be the set of varieties $X$ such that
	\begin{enumerate}
		\item $X$ is a rationally connected klt Calabi--Yau $3$-fold, and
		\item $\tilde{q}(X) >0$.
	\end{enumerate}
	Then $\mathcal{D}_{ab}$ is bounded modulo flops.
\end{cor}

\begin{proof}
	
	{\bf Step 1}. We will construct a quasi-\'etale Galois cover $\gamma:Y\to X=Y/G$ such that $Y$ is either an abelian $3$-fold or a product of an elliptic curve 
	and a weak K3 surface (i.e., a normal projective surface $S$ with canonical singularities, $K_S \sim 0$ and $\tilde{q}(S)=0$, cf. \cite{MR2587100}).

	Let $\gamma\colon W_1 \to X$ be the quasi-\'etale
	cover given in Theorem \ref{t.ggk}.
	Since $\tilde{q}(X) >0$,
	it follows that $W_1$ is either an abelian $3$-fold or a product of an elliptic curve 
	and a weak K3 surface. Note that $W_1$ has canonical singularities.
	
	
	In view of \cite[Theorem 3.7]{gkp},
	there exists a finite surjective morphism $\phi_1 \colon Y_1 \to W_1$
	such that $\gamma_1=\gamma\circ  \phi_1$ is Galois and the branch loci of
	$\gamma$ and $\gamma_1$ are the same. The latter property implies
	that both $\gamma_1$ and $\phi_1$ are quasi-\'etale.

	If $W_1$ is an abelian $3$-fold, then by purity of the branch locus it 
	follows that $\phi_1$ is \'etale and hence  $Y_1$ is also an
	abelian $3$-fold. In this case we complete Step 1 by setting $Y=Y_1$.
	
	Now assume that $W_1$ is a product of an elliptic curve 
	and a weak K3 surface. Note that $\tilde{q}(W_1)=1$. 
	Now we construct, by induction, a sequence of finite quasi-\'etale surjective morphism
	\begin{displaymath}
	\xymatrix{
		\cdots \ar[r] & Y_i \ar[r]^{\gamma_i} & Y_{i-1} \ar[r]^{\gamma_{i-1}} & \cdots 
		\ar[r]^{\gamma_2} & Y_1 \ar[r]^{\gamma_1} &Y_0=X,
	}
	\end{displaymath} 
	such that for each $i$, the compositions $Y_i\to X$ are Galois, and $\gamma_i: Y_i\to Y_{i-1}$ factors as 
	\begin{displaymath}
	\xymatrix{
		Y_i \ar[r]^{\phi_i} & W_{i} \ar[r]^{\psi_{i}} & Y_{i-1},}
	\end{displaymath} 
	where $W_i$ is a product  of an elliptic curve 
	and a weak K3 surface. Moreover, $\psi_i$ is \'etale if $i>1$.
	
	The construction is as follows: assume that we have constructed $Y_{i-1}$, then $K_{Y_{i-1}}\sim 0$ and  $Y_{i-1}$ has canonical singularities. By \cite[Theorem 8.3, Corollary 8.4]{Kawamata85}, there exists an \'etale covering $\psi_i: W_{i} \to Y_{i-1}$ where $W_i=C\times Z$, $C$ is an abelian variety and $Z$ is a canonical variety with $K_Z\sim 0$. We only need to show that in our case, $C$ is an elliptic curve and $Z$ is a surface with $\tilde{q}(Z)=0$. Assume not, then it is easy to see that $\tilde{q}{(C\times Z)}>1=\tilde{q}(W_1)$, which is absurd since $\tilde{q}$ is invariant under quasi-\'etale covers.
	Then  $\psi_i: W_{i} \to Y_{i-1}$ is constructed by applying  Theorem \ref{t.ggk}, and $\phi_i: Y_i \to W_{i}$ is constructed by applying \cite[Theorem 3.7]{gkp} to the composition $W_i\to X$.
	
	By \cite[Theorem 1.1]{gkp}, it follows that there exists $i>1$ such 
	that $\gamma_i$ is \'etale. 
	It follows that
	the morphism $Y_i \to W_i$ is \'etale (see \cite[Corollary 3.6]{milne}). We may write $W_i=E\times S$. As $S$ is simply connected,
	see \cite[Theorem 7.8]{k.shaf}, it follows that $Y_i \simeq E' \times S$ for some
	elliptic curve $E'$. 
	Hence we complete Step 1 by taking $Y=Y_i$.
	
	{\bf Step 2}. We will show that the group action of $G$ on $Y$ is bounded.
	
	As $X=Y/G$ and $Y$ is Gorenstein, if we bound the order 
	of the representation of $G$ on the vector space $H^0(Y, K_Y)$ by a universal
	constant $C$, then the index of $K_X$ is at most $C!$.
	Boundedness modulo flops of $\mathcal{D}_{ab}$ then follows from Theorem \ref{t.bdd.mld}, as the discrepancies of $X$ will all be contained in $\frac{1}{C!} \mathbb{Z}$.

	{\bf Step 2.1}.
	When $Y$ is an abelian $3$-fold, the above claim follows immediately
	from the fact that $\mathrm{rk} \;H^3(Y, \mathbb{Z}) = 20$ and for any $g \in G, \; 
	g^\ast$ is an automorphism of $H^3(Y, \mathbb{Z})$ defined over the 
	integers, hence its minimal polynomial has degree $\leq 20$. That
	implies that any eigenvalue of $g^\ast$ is a root of unity of bounded
	order. 
	
	{\bf Step 2.2}.
	When $Y$ is a product of an elliptic curve $E$ and 
	a weak K3 surface $S$ the claim follows in the same vein.
	In fact, take $S'\to S$ the minimal resolution, then $S'$ is a K3 surface, and $g$ lifts to an automorphism of $E\times S'$. This
	implies that any eigenvalue of $g^\ast$ is a root of unity of bounded
	order since $\mathrm{rk} \;H^3(E\times S' , \mathbb{Z}) = 44$.
\end{proof}

\appendix
\section{On rationally connectedness of varieties of CY-type}\label{appendix}

The goal of this appendix is to give a sufficient condition for a dlt log Calabi--Yau pair to be rationally connected. It was suggested by Chenyang Xu and carried out during a discussion with Zhiyu Tian.

Given a log pair $(X, B)$, recall that a subvariety $C\subset X$ is said to be a {\it log center} or {\it non-canonical center}  if there is a prime divisor $E$ over $X$ with center $C$ such that the log discrepancy $a(E, X, B)<1$.

The following is the main result of this appendix.

\begin{thm}\label{thm:dltcyrc}
	Let $(X,B)$ be an lc log Calabi--Yau pair with a log center $C$, 
	then $X$ is rationally chain connected modulo $C$.
	In particular, if $C$ is rationally chain connected, then  $X$ is rationally chain connected.
\end{thm}

As a simple corollary, we have the following result. 
\begin{cor}\label{cor:dltcydisrc}
	Let $(X,B)$ be a dlt log Calabi--Yau pair with a $0$-dimensional log center. Then  $X$ is rationally connected.
\end{cor}


\begin{proof}[Proof of Theorem \ref{thm:dltcyrc}]
	By assumption, there exists a prime divisor $E$ over $X$, such that the center of $E$ on $X$ is $C$, and $a:=a(E,X,B)<1$. By \cite[Corollary 1.4.3]{BCHM}, after taking a $\bQ$-factorial dlt modification of $(X, B)$, there exists a birational morphism $f:X'\to X$, such that
	$$K_{X'}+B'+(1-a)E=f^{*}(K_X+B)\equiv0,$$
	where $B'$ is the sum of the birational transform of $B$ on $X'$ and the exceptional divisors of $f$ except for $E$. It suffices to show that $X'$ is rationally chain connected modulo $E$.
	
	Run a $(K_{X'}+B')$-MMP with scaling of an ample divisor on $X'$, according to \cite[Corollary 1.3.3]{BCHM}, this MMP ends up with a Mori fiber space $ Y\to Z$.
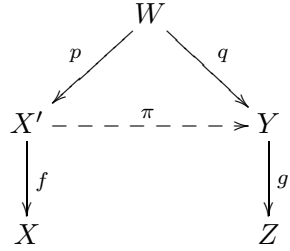
	$$\xymatrix@=2.5em{
		&        W  \ar[dl]_{p}  \ar[dr]^{q}                &\\	
		X'\ar@{-->}[rr]^{\pi}\ar[d]^{f}  & &   Y\ar[d]^{g}\\
		X         &  &Z\\
	}
	$$ 
	Let $W\subset X'\times Y$ be the closure of the graph of $\pi$, and $p$ and $q$ are the projections from $W$ to $X'$ and $Y$, respectively. Since $-K_Y$ is ample over $Z$ and $Y$ is klt,  according to \cite[Theorem 1.2]{hm07}, every fiber of $W\to Z$ is rationally chain connected. 
	
	Let $E_Y,E_W$ be the strict transforms of $E$ on $Y$ and $W$, respectively. Since this MMP is also a $(-E)$-MMP, $E_Y$ dominates $Z$ and so does $E_W$. Thus, $W$ is rationally chain connected modulo $E_W$, and hence $X'$ is rationally chain connected modulo $E$. We complete the proof.
\end{proof}
\begin{proof}[Proof of Corollary \ref{cor:dltcydisrc}]
	Follows easily from Theorem \ref{thm:dltcyrc} and \cite[Corollary 1.5(2)]{hm07}.
\end{proof}

\end{document}